\documentclass[12pt]{amsart}

\usepackage{amssymb,latexsym}

\usepackage{manfnt}

\usepackage{subfigure}

\usepackage{enumerate}

\usepackage{lmodern}

\usepackage[T1]{fontenc}

 \usepackage[french,english]{babel}

 \usepackage{hyperref}
 
\hypersetup{
    colorlinks=true,
    linkcolor=blue,
    filecolor=blue,      
    urlcolor=blue,
    citecolor=blue,
    pdftitle=,
    pdfpagemode=FullScreen,
    }

\makeatletter

\@namedef{subjclassname@2010}{

  \textup{2010} Mathematics Subject Classification}

\makeatother
\newtheorem{thm}{Theorem}[section]

\newtheorem{lem}[thm]{Lemma}
\newtheorem{pro}[thm]{Proposition}
\theoremstyle{definition}
\newtheorem{defin}{Definition}
\newtheorem{rem}[thm]{Remark}

\numberwithin{equation}{section}

\newcommand{\X}{\mathbb{X}}
\newcommand{\T}{\mathcal{T}}
\newcommand{\Y}{\mathbb{Y}}

\newcommand{\ex}{\mathbb{E}}
\newcommand{\re}{\textup{Re}}

\newcommand{\pr}{\mathbb{P}}
\newcommand{\ep}{\varepsilon}

\newcommand{\F}{\mathcal{F}_3(Q)}

\newcommand{\sums}{\sideset{}{^*}\sum}

\newcommand{\newabstract}[1]{%
  \par\bigskip
  \csname otherlanguage*\endcsname{#1}%
  \csname captions#1\endcsname
  \item[\hskip\labelsep\scshape\abstractname.]
}

\begin{document}

\baselineskip=17pt

\title{The distribution of the maximum of cubic character sums}

\author{Youness Lamzouri}
\author{Kunjakanan Nath}

\address{
Universit\'e de Lorraine, CNRS, IECL,  and  Institut Universitaire de France,
F-54000 Nancy, France}
\email{youness.lamzouri@univ-lorraine.fr}

\address{
Universit\'e de Lorraine, CNRS, IECL,  
F-54000 Nancy, France}
\email{kunjakanan@gmail.com}

\date{}

\begin{abstract}
For a primitive Dirichlet character $\chi\pmod q$ we let 
\[M(\chi):= \frac{1}{\sqrt{q}}\max_{1\leq t \leq q} \Big|\sum_{n \leq t} \chi(n) \Big|.\]
In this paper, we investigate the distribution of $M(\chi)$, as $\chi$ ranges over primitive cubic characters $\chi\pmod q$ with $(q,3)=1$ and $q\leq Q$. 
Our first result gives an estimate for the proportion of such characters for which $M(\chi)>V$, in a uniform range of $V$, which is best possible under the assumption of the Generalized Riemann Hypothesis. In particular, we show that the distribution of large cubic character sums behaves very differently from those in the family of non-principal characters modulo a large prime, and the family of quadratic characters. We also investigate the location of the number $N_{\chi}$ where the maximum of $|\sum_{n\leq N} \chi(n)|$ is attained, and show the surprising result that for almost all primitive cubic characters $\chi\pmod q$ with  $M(\chi)>V$, $N_{\chi}/q$ is very close to a reduced fraction with a large denominator of size $(\log V)^{1/2+o(1)}$. This contradicts the common belief that for an even character $\chi$, $N_{\chi}/q$ is located near a rational of small denominator and gives a striking difference with the case of even characters in the other two families mentioned above, for which $N_{\chi}/q\approx 1/3$ or $2/3$ for almost all even $\chi$. Furthermore, in the case of cubic characters, the works of Granville-Soundararajan, Goldmakher, and Lamzouri-Mangerel show that if $M(\chi)$ is large, then $\chi$ pretends to be $\xi(n)n^{it}$ for some small $t$, where $\xi$ is an odd character of small conductor $m$. We show that for almost all such characters, we have $M(\chi)=m^{-1/2+o(1)}\big|L(1+it, \chi\overline{\xi})\big|.$   

\end{abstract}


\thanks{The first author is supported by a junior chair of the Institut Universitaire de France.}

\maketitle

\setcounter{tocdepth}{1}



\section{Introduction}

\subsection{Bounds on long character sums} Dirichlet characters and sums involving them have been extensively studied over the last century. Let $\chi$ be a non-principal Dirichlet character modulo $q$. An important quantity associated with $\chi$ is\footnote{Throughout this paper, we will normalize the maximum of the character sum $\sum_{n \leq t} \chi(n)$ by $\sqrt{q}$. Such a normalization is natural in view of the P\'olya Fourier expansion \eqref{Eq:PolyaFourierExpansion}.}
$$M(\chi) := \frac{1}{\sqrt{q}}\max_{t \leq q} \Big|\sum_{n \leq t} \chi(n) \Big|.
$$
It is known that $M(\chi)\gg 1$ (see \cite[Theorem 9.23]{MV07}) for all primitive characters modulo $q$. On the other hand, the best known upper bound for $M(\chi)$ is the classical P\'olya-Vinogradov inequality (see \cite[Theorem 9.18]{MV07}), which dates back to 1918, and states that
\begin{equation}\label{Eq:PolyaVinogradov}
M(\chi)  \ll  \log q.
\end{equation}
It is amazing that this century-old result, which is proved using only basic Fourier analysis, has resisted substantial improvement outside of certain special cases. In fact, an improved bound of the form $M(\chi)=o(\log q)$ would have immediate consequences on other problems in analytic number theory (see, for example, \cite{BoGo16, GrMa23, Ma20}). Almost sixty years later, Montgomery and Vaughan \cite{MV77} succeeded in improving the P\'olya-Vinogradov inequality, assuming the Generalized Riemann Hypothesis (GRH). More precisely they showed that under GRH one has 
\begin{equation}\label{Eq:MVCHARBOUND}
M(\chi)\ll \log\log q.
\end{equation}
Although conditional, this result is significant since the Montgomery-Vaughan bound is best possible in view of an old result of Paley \cite{Pa32}, who proved the existence of an infinite family of primitive quadratic characters $\chi \pmod q$ for which $M(\chi)\gg \log\log q$. 

In 2007, Granville and Soundararajan \cite{GrSo07} made a major breakthrough by showing that both the P\'{o}lya-Vinogradov and the Montgomery-Vaughan bounds can be improved for characters of a fixed odd order. In particular, in the case of a primitive cubic character $\chi \pmod q$, they showed that 
\begin{equation}\label{Eq:UpperGrSo}
M(\chi)\ll \begin{cases} (\log q)^{1/2+3\sqrt{3}/(4\pi)+o(1)} & \text{ unconditionally,}\\
 (\log\log q)^{1/2+3\sqrt{3}/(4\pi)+o(1)} & \text{ assuming GRH.}\end{cases}
\end{equation} 
 By refining their method, Goldmakher \cite{Gol} obtained the better exponent $3\sqrt{3}/(2\pi)$ in both results. Moreover, Goldmakher and Lamzouri \cite{GoLa12} proved that the conditional part of Goldmakher's result is optimal up to $(\log\log q)^{o(1)}$. More precisely, they showed the existence of an infinite family of primitive cubic characters $\chi \pmod q$ such that 
 $$ M(\chi)\gg_{\ep} (\log\log q)^{3\sqrt{3}/(2\pi)-\ep}.$$
To simply our notation throughout we put 
\begin{align}\label{Def:beta}
    \beta:=\frac{3\sqrt{3}}{2\pi}=0.826993\dotsc.
\end{align}
By using a new Hal\'asz-type inequality for logarithmic mean values of completely multiplicative functions,  Lamzouri and Mangerel pushed these results further in \cite{LaMa22}. Indeed, they first proved that 
\begin{equation}\label{Eq:UpperLamzouriMangerel}
   M(\chi)\ll \begin{cases} (\log q)^{\beta}(\log_2 q)^{-1/4+o(1)} & \text{ unconditionally,}\\
 (\log_2 q)^{\beta}(\log_3 q)^{-1/4}(\log_4 q)^{O(1)} & \text{ assuming GRH,}\end{cases} 
\end{equation} where here and throughout we let $\log_k$ denote the $k$-th iterate of the natural logarithm function. Furthermore, Lamzouri and Mangerel proved the existence of an infinite family of primitive cubic characters $\chi \pmod q$ such that 
 $$ M(\chi)\gg (\log_2 q)^{\beta}(\log_3 q)^{-1/4}(\log_4 q)^{O(1)}.$$

\subsection{The distribution of $M(\chi)$ over families of characters}
Although $M(\chi)$ can grow as large as $\log\log q$, Montgomery and Vaughan \cite{MV79} showed that this is a rare event and that $M(\chi)\ll 1$ for most characters. To this end, they proved that for any fixed positive real number  $k$, the $k$-th moment of $M(\chi)$ over the family of primitive characters $\chi \pmod q$ is bounded by a constant depending on $k$. Their results can be stated in terms of the distribution function
$$ \Phi_q(V):=\frac{1}{\varphi(q)} \# \{\chi \ (\bmod \ q) : M(\chi)>V\},$$
where $\varphi(q)$ is Euler's totient function.
Indeed, it follows from the work of Montgomery and Vaughan that 
$$ \Phi_q(V) \ll_C V^{-C}, $$ for any constant $C\geq 1$. 
This estimate was improved by Bober and Goldmakher \cite{BoGo13}, who showed that for fixed $V$,  $\Phi_q(V)$ decays exponentially faster. Subsequently, Bober, Goldmakher, Granville, and Koukoulopoulos \cite{BGGK18} proved a much stronger and uniform result in the case where $q$ is a large prime. More specifically, they showed that in the range $1\ll V \leq (e^{\gamma}/\pi+o(1))\log\log q$ one has\footnote{Here $\gamma$ is the Euler-Mascheroni constant.} 
\begin{equation}\label{Eq:BGGK}
\Phi_q(V)= \exp\left(-\frac{e^{e^{-\gamma}\pi V+O(1)}}{V}\right). 
\end{equation}
 Their method relies on the exact orthogonality relation for characters modulo $q$ and does not seem to generalize to other families of Dirichlet characters.
Using a different approach, based on the quadratic large sieve, Lamzouri \cite{La22} obtained a similar result for the family of quadratic characters. More precisely, he showed that the estimate \eqref{Eq:BGGK} holds for the proportion of fundamental discriminants $|d|\leq Q$ such that $M(\chi_d)>V$ in a similar range of uniformity $1\ll V\leq (e^{\gamma}/\pi+o(1))\log\log Q$, if $Q$ is large. Here $\chi_d= \left(\frac{d}{\cdot}\right)$ is the Kronecker symbol modulo $|d|$. Lamzouri also proved analogous results for the family of Legendre symbols modulo primes $p\leq Q.$

Although the family of primitive characters modulo $q$, and that of quadratic characters attached to fundamental discriminants have been extensively studied in various contexts (large sieve inequalities, moments and non-vanishing of $L$-functions at the central point, distribution of values of $L$-functions,  low-lying zeros, zero density estimates, to name a few), up until recently much fewer results have been established for cubic Dirichlet characters, due to several conceptual and technical difficulties. Here and throughout, for a large real number Q, we shall denote
\[\F:=\{\chi\bmod q\colon \chi\: \text{primitive and cubic with $(q, 3)=1$ and $q\leq Q$}\}.\]
We also put $\omega_3:=\exp(2\pi i/3)$. The main difficulty when dealing with cubic characters lies in the fact that the natural setting where such characters should be studied is over the ring of Eisenstein integers $\mathbb{Z}[\omega_3]$, due in particular to the cubic reciprocity law. An important example is Heath-Brown's celebrated cubic large sieve inequality \cite{HB00} which, unlike his quadratic large sieve over the integers, is stated over $\mathbb{Z}[\omega_3]$, requires the inner summation to be taken over squarefree integers, and has an additional mysterious term of size $(MN)^{2/3}$. It was thought for a long time that this term could be removed until the very recent groundbreaking work of Dunn and Radziwi\l\l\: \cite{DuRa}  on Patterson's conjecture for cubic Gauss sums, which showed that the cubic large sieve is optimal under the assumption of the Generalized Riemann Hypothesis. Baier and Young \cite{BaYo10} have managed to somehow transfer the  cubic large sieve inequality of Heath-Brown to $\mathbb{Z}$ 
but their result is complicated and much less efficient than the quadratic large sieve. In the same paper, they proved an asymptotic formula for the first moment of $L(1/2, \chi)$ as $\chi$ varies in $\F$, and used their large sieve inequality to obtain an upper bound on the second moment, where $L(s, \chi)$ is the Dirichlet $L$-function attached to $\chi$. It is worth noting here that unlike the family of primitive characters modulo $q$ and that of quadratic characters, no asymptotic formula is known for the second moment of $L(1/2, \chi)$ over $\F$. 
We should also mention a very recent work of Darbar, David, Lal\'in, and Lumley \cite{DDLL23} on the distribution of values at $1$ of Dirichlet $L$-functions attached to cubic characters in $\F$. Using ideas from the works of Granville-Soundararajan \cite{GrSo03} and Dahl-Lamzouri \cite{DaLa18}, they constructed a probabilistic random model for this family and estimated the distribution of large and small values of $L(1, \chi)$, as $\chi$ varies in $\F$.  In particular, they proved the interesting fact that the distribution of the small values of $L(1, \chi)$ over cubic characters behaves differently from the other two studied families of characters, namely primitive characters modulo $q$ and quadratic characters.

The goal of this paper is to investigate the distribution of large values of $M(\chi)$, as $\chi$ varies in the family of primitive cubic characters with conductor less than $Q.$ As is expected from the upper bound \eqref{Eq:UpperGrSo} of Granville and Soundararajan, and the further refinements by Goldmakher and Lamzouri-Mangerel, we prove that the proportion of cubic characters $\chi\in \F$ for which $M(\chi)>V$ decays much faster than \eqref{Eq:BGGK}.    
\begin{thm}\label{Thm:Main}
Let $Q$ be large. Let $\ep>0$ be a small fixed real number. There exists a positive constant $C_0$ such that uniformly for $V$ in the range 
\[C_0\leq V\leq (\log_2 Q)^{\beta}(\log_3 Q)^{-1/4-\ep},\]
we have 
$$ \frac{1}{\#\F}\#\{\chi \in \F\colon M(\chi)>V\}= \exp\left(-\exp\left(V^{1/\beta}(\log V)^{1/(4\beta)+o(1)}\right)\right),$$
where $\beta$ is given by \eqref{Def:beta}.
\end{thm}

\begin{rem}
We first note that  \eqref{Eq:UpperLamzouriMangerel} shows that the range of validity of Theorem \ref{Thm:Main} is best possible up to $(\log_3 Q)^{o(1)}$, under the assumption of the Generalized Riemann Hypothesis. In fact, we obtain a better result for the implicit lower bound in Theorem \ref{Thm:Main}. More precisely, our proof shows the existence of a positive constant $C>0$ such that if \[C_0\leq V\leq (\log_2 Q)^\beta (\log_3 Q)^{-1/4}(\log_4 Q)^{-C},\]
then
\begin{align*}
    \frac{1}{\#\F}\#\{\chi \in \F \colon M(\chi)>V\}\geq \exp\left(-\exp\left(V^{1/\beta}(\log V)^{1/(4\beta)}(\log_2 V)^{O(1)}\right)\right).
\end{align*}
\end{rem}
We are unable to obtain a similar upper bound because we don't know how to rule out the existence of Landau-Siegel exceptional zeros. Indeed, an additional term coming from a hypothetical exceptional modulus $m\leq (\log_2 Q)^{O(1)} $, discovered by Lamzouri and Mangerel in \cite{LaMa22}, is responsible for the term $(\log V)^{o(1)}$ in the estimate for the distribution function in Theorem \ref{Thm:Main}. 

\begin{rem} 
Our approach is flexible and could be adapted to yield the implicit upper bound in Theorem \ref{Thm:Main} for any family of Dirichlet characters $\mathcal{F}$ with conductors $q\leq Q$ and cardinality $\gg Q^{\delta}$ (for some $\delta>0$), provided that $\mathcal{F}$ satisfies an appropriate ``orthogonality relation'' with a power-saving error term for small $n$ (as in \cite[Lemma 4.1]{GrSo03} for quadratic characters, or \cite[Lemma 3.5]{DDLL23} for cubic characters). In particular, for the family of primitive Dirichlet characters of prime order $g> 3$ and conductor $q\leq Q$, one would need to generalize Lemma 3.5 of \cite{DDLL23} to this setting. This appears to be non-trivial in general, as the ring of integers of $\mathbb{Q}(e^{2\pi i/g})$ may not be a principal ideal domain. On the other hand, in order to establish the implicit lower bound in Theorem \ref{Thm:Main} for a family of characters $\mathcal{F}$, it is necessary to obtain precise estimates for the distribution of large values of $L(1,\chi)$ as $\chi$ ranges over $\mathcal{F}$. This was accomplished by Granville and Soundararajan \cite{GrSo03} in the case of quadratic characters, and was recently extended to cubic characters by Darbar, David, Lal\'in, and Lumley \cite{DDLL23}. Although some results on the extreme values of $L(1,\chi)$ are known for characters of order $g$, for any fixed $g\ge 4$ (see, for example, \cite{La17}), the available estimates on the distribution of these values are not sufficiently strong to yield the analogue of the lower bound in Theorem~\ref{Thm:Main}. In summary, it would be interesting to extend our results to the family of $g$-th order Dirichlet characters for any $g\geq 4$.
    \end{rem}

\subsection{Key ideas in the proof of Theorem \ref{Thm:Main}}
 Our proof of the upper bound of Theorem \ref{Thm:Main} combines ingredients from the work of Lamzouri and Mangerel \cite{LaMa22}, which relies on the pretentious theory of character sums developed by Granville and Soundararajan \cite{GrSo07}, with the recent work of Lamzouri \cite{La22} on the distribution of quadratic character sums.  However, several new technical difficulties arise in our case, notably with respect to the cubic large sieve. Indeed, the method of Lamzouri in \cite{La22}
is based on the quadratic large sieve, and it turns out that the cubic large sieve of Baier-Young \cite{BaYo10} (or that of Heath-Brown \cite{HB00}) is not suitable for our purposes. To overcome this problem, we replace the cubic large sieve by the classical large sieve (over all primitive characters $\chi\pmod q$ with $q\leq Q$) to show that the portion of the tail in the P\'olya Fourier expansion of a cubic character $\chi$ (see \eqref{Eq:PolyaFourierExpansion} below) from $Q^{\ep}$ to $Q^{1/2+\delta}$ is small very often. We then handle the remaining range (which we further divide into several pieces) using a new simple cubic large sieve inequality 
which follows from \cite{DDLL23}. This is a consequence of cubic reciprocity and the P\'olya-Vinogradov inequality for Hecke characters over $\mathbb{Q}(\omega_3)$. A weaker version for cubic Dirichlet characters with prime moduli was obtained by  Elliott \cite[Lemma 33]{El70}.    
Now, to prove the lower bound of Theorem \ref{Thm:Main} we combine ideas from the work of Lamzouri and Mangerel with the recent work of Darbar, David, Lal\'in, and Lumley \cite{DDLL23} on the distribution of $L(1, \chi)$ as $\chi$ varies in $\F$.
 \subsection{A structure result for large cubic character sums}
Granville and Soundararajan \cite{GrSo07} showed that if $M(\chi)$ is large then $\chi$ must ``pretend'' to be a character $\xi_{\chi}$ of small conductor and opposite parity.  Building on their work, Bober, Goldmakher, Granville, and Koukoulopoulos \cite{BGGK18} proved that for almost all primitive characters $\chi\pmod q$ with $M(\chi)$ large, $\xi_{\chi}$ is the trivial character if $\chi$ is odd, and $\xi_{\chi}= \left(\frac{\cdot}{3}\right)$ if $\chi$ is even. Similar results were obtained for the family of quadratic characters by Dong, Wang, and Zhang \cite{DWZ23}, by combining the ideas of \cite{BGGK18} with the work of Lamzouri \cite{La22}.
In the case of cubic characters $\chi$, we prove however that $\xi_{\chi}$ must have a relatively large conductor. Indeed, our Theorem \ref{Thm:Structure} below shows that for almost all cubic characters $\chi\in \F$ with $M(\chi)>V$, $\chi$ pretends to be $\xi(n)n^{it}$ (for some small $t$) where $\xi$ is an odd character of conductor $m=(\log V)^{1/2+o(1)}.$ We also show that for these cubic characters we have  $$M(\chi)=m^{-1/2+o(1)} \big|L(1+it, \chi\overline{\xi})\big|.$$

For a primitive character $\chi\pmod q$, we denote by $N_{\chi}$ a point at which the maximum in $M(\chi)$ is attained (that is $M(\chi)=|\sum_{n\leq N_{\chi}}\chi(n)|/\sqrt{q}$) and put $\alpha_{\chi}:=N_{\chi}/q$. It follows from the work of Montgomery and Vaughan \cite{MV79} that if $\chi$ is even and $M(\chi)$ is large then\footnote{This is different in the case of odd characters because of the extra factor $\sum_{1\leq |n|\leq q^{1/2+\delta}}\chi(n)/n\approx 2L(1,\chi) $ in the P\'olya Fourier expansion \eqref{Eq:PolyaFourierExpansion} for such characters.}  $\alpha_{\chi}$ must be close to a rational with a small denominator. Bober, Goldmakher, Granville, and Koukoulopoulos investigated the location of $N_{\chi}$ in \cite{BGGK18}, and showed the interesting fact that for almost all primitive even characters $\chi\pmod q$, one has $N_{\chi}/q\approx 1/3$ or $2/3$. This is also the case for even quadratic characters (see \cite{DWZ23}). In the case of cubic characters, however, we prove the surprising fact that for almost all $\chi\in \F$ with $M(\chi)>V$, $N_{\chi}/q$ is very close to a reduced fraction $a/b$ with a relatively large denominator $b=(\log V)^{1/2+o(1)}.$

In order to state our result, we first need some notation. Let $\ep>0$ be small and fixed, and $Q$ be large. Let $V$ be in the range $C_1\leq V\leq (\log_2 Q)^{\beta}(\log_3 Q)^{-1/4-\ep}$, for some suitably large constant $C_1>0$. For such a $V$ we put 
\begin{equation}\label{Eq:DefinitionYV}
y:=y(V)=\exp(V^{1/\beta}(\log V)^{1/(4\beta)+\ep/2}).  
\end{equation}
For a primitive character $\chi\pmod q$ we let $t=t_{\chi}\in [-(\log y)^{-\frac{7}{11}}, (\log y)^{-\frac{7}{11}}]$ and 
$\xi=\xi_{\chi} \pmod m$ be the primitive character with conductor less than $(\log y)^{4/11}$, which minimize the quantity  $$\sum_{p\leq y}\frac{1-\re(\chi(p)\overline{\xi(p)})}{p^{1+it}}.$$ 
\begin{thm}\label{Thm:Structure}
With the same notation as above there exists a set $\mathcal{C}_Q(V)\subset \{ \chi\in \F \ : \ M(\chi)>V \}$ such that 
$$ \# \mathcal{C}_Q(V)= \left(1+O\left(\exp\left(-\exp(V^{1/\beta}(\log V)^{1/(4\beta)})\right)\right)\right)\#\left\{ \chi\in \F : \ M(\chi)>V \right\},$$ and for all $\chi\in \mathcal{C}_Q(V)$ we have
\begin{itemize}
\item[1.] If $a/b$ is the reduced fraction for which $|\alpha_{\chi}-a/b|\leq 1/(bB)$ with $b\leq B:=e^{V/\log V}$, then $b=(\log V)^{1/2+O(\ep)}$.
    \item[2.] $\xi$ is odd, $m\mid b$  and $m=(\log V)^{1/2+O(\ep)}.$
    \item[3.] We have 
    $$ M(\chi)=\frac{1}{m^{1/2+O(\ep)}} \left|L(1+it, \chi\overline{\xi})\right|.$$
\end{itemize}

\end{thm}

The proof of Theorem \ref{Thm:Structure} combines ingredients from the work of Lamzouri and Mangerel \cite{LaMa22}, together with our Theorem \ref{Thm:TailPolya} below, which bounds the frequency of how large can the  tail in the P\'olya Fourier expansion of a cubic character twisted by $n^{it}$ (uniformly for $t\in [0,1]$) be.

 \subsection{Notation} We will use standard notations in this paper. However, for the convenience of readers, we would like to highlight a few of them. For any real number $x$, we write $e(x)$ to denote $e^{2\pi i x}$. Expressions of the form ${f}(x)=O({g}(x))$, ${f}(x) \ll {g}(x)$, and ${g}(x) \gg {f}(x)$ signify that $|{f}(x)| \leq C|{g}(x)|$ for all sufficiently large $x$, where $C>0$ is an absolute constant. A subscript of the form $\ll_A$ means the implied constant may depend on the parameter $A$. The notation ${f}(x) \asymp {g}(x)$ indicates that ${f}(x) \ll {g}(x) \ll {f}(x)$. Next, we write $f(x)=o(g(x))$ if $\lim_{x\to \infty}f(x)/g(x)=0$. Furthermore, given a positive integer $n$, we let $P^+(n)$ and $P^-(n)$ be the largest and smallest prime divisors of $n$, with the standard convention that $P^+(1)=1$ and $P^-(1)=\infty$. 
The letter $p$ always denotes a prime number, unless otherwise mentioned. Moreover, the various constants $c$'s appearing in the proofs may not all be the same.


 \section{Key components of the proof of Theorem \ref{Thm:Main}} \label{Sec: thmmain}
 The first step in the proof of Theorem \ref{Thm:Main}
 is to use the P\'olya Fourier expansion for a primitive character $\chi\pmod q$ (see \cite[eqn. (9.19), p.311]{MV07}), which gives
 \begin{equation}\label{Eq:PolyaFourierExpansion}
   \sum_{n\leq t}\chi(n)=\dfrac{\tau(\chi)}{2\pi i}\sum_{1\leq |n|\leq Z}\dfrac{\overline{\chi}(n)}{n}\bigg(1-e\bigg(\dfrac{-nt}{q}\bigg)\bigg) + O\bigg(1+\dfrac{q\log q}{Z}\bigg),  
 \end{equation}
 where $\tau(\chi)$ is the Gauss sum attached to $\chi$ and $Z\geq 1$ is a parameter to be chosen appropriately. Since we are restricting ourselves to cubic characters $\chi$, we have $\chi(-1)=1$. This implies  $\sum_{1\leq |n|\leq Z}\chi(n)/n=0$ and hence 
\begin{equation}\label{Eq:PolyaFourier2}\sum_{n\leq t}\chi(n)=-\dfrac{\tau(\chi)}{2\pi i}\sum_{1\leq |n|\leq Z}\dfrac{\overline{\chi}(n)}{n}e\bigg(\dfrac{-nt}{q}\bigg) + O\bigg(1+\dfrac{q\log q}{Z}\bigg).
 \end{equation}
  Choosing $Z:=Q^{21/40}$ and using the fact that $|\tau(\chi)|=\sqrt{q}$ (see \cite[Theorem 9.7]{MV07}), we deduce
 \begin{align}\label{Eq:RefinedPolyaFourierExpansion}
     M(\chi)=\dfrac{1}{2\pi}\max_{\alpha\in [0,1)}\bigg|\sum_{1\leq |n|\leq Z}\dfrac{\chi(n)e(n\alpha)}{n}\bigg| + O(1).
 \end{align}
Next, following Montgomery-Vaughan \cite{MV77}, we decompose the above sum on the right-hand side into two parts so that
\begin{align}\label{Eq:MV}
     M(\chi)\leq \dfrac{1}{2\pi}\max_{\alpha\in [0,1)}\bigg|\sum_{\substack{1\leq |n|\leq Z\\P^+(n)\leq y}}\dfrac{\chi(n)e(n\alpha)}{n}\bigg| + \dfrac{1}{\pi}\max_{\alpha\in [0,1)}\bigg|\sum_{\substack{1\leq n\leq Z\\P^+(n)> y}}\dfrac{\chi(n)e(n\alpha)}{n}\bigg| + O(1),
\end{align}
where the parameter $y$ will be chosen appropriately. We will show, similarly to \cite{BGGK18} and  \cite{La22}, that the bulk of the contribution in the above sum comes from the smooth part:
\begin{equation}\label{Eq:SmoothPart}
\max_{\alpha\in [0,1)}\bigg|\sum_{\substack{1\leq |n|\leq Z\\P^+(n)\leq y}}\dfrac{\chi(n)e(n\alpha)}{n}\bigg|.
\end{equation}
More precisely, the implicit upper bound in Theorem \ref{Thm:Main} will follow from the following two results. The first is an upper bound on the smooth 
part \eqref{Eq:SmoothPart}.
 \begin{thm}\label{Thm: smooth}
Let $Q$ be large. Let $Z=Q^{21/40}$ and $\chi\in \mathcal{F}_3(Q)$. If $y\leq \log Q$, then 
\begin{align*}
    \max_{\alpha\in [0,1)}\bigg|\sum_{\substack{1\leq |n|\leq Z\\ P^+(n)\leq y}}\dfrac{\chi(n)e(n\alpha)}{n}\bigg|\ll (\log y)^{\beta}(\log_2y)^{-1/4+o(1)},
\end{align*}
where $\beta$ is given by \eqref{Def:beta}.
    
\end{thm}
Our next result shows that the second part in \eqref{Eq:MV} is bounded for most cubic characters $\chi\in \F$. In fact, we prove a stronger result where we also twist by $n^{it}$, uniformly for $t\in [0, 1]$. This will be useful in the proof of our structure result Theorem \ref{Thm:Structure}.
\begin{thm}\label{Thm:TailPolya}
Let $h$ be a completely multiplicative function such that $|h(n)|\leq 1$ for all integers $n\geq 1$. Let $Q$ be large and  $Z=Q^{21/40}$. There exist positive constants $b_1, b_2$, and $b_3$  such that for all real numbers $b_1\leq y\leq \log Q/(b_2\log_2Q)$,  the number of characters $\chi\in \F$ such that
$$ \max_{t\in [0,1]}\max_{\alpha\in [0, 1)}\Bigg|\sum_{\substack{1\leq n\leq Z\\ P^{+}(n)> y}} \frac{\chi(n) h(n) e(n\alpha)}{n^{1+it}}\Bigg|> 1$$
is 
$$ \ll Q\exp\left(-b_3\frac{y}{\log y}\right).$$
\end{thm}

We will prove Theorems \ref{Thm: smooth} and \ref{Thm:TailPolya} in Sections \ref{Sec: smooth part} and \ref{Sec: tail}, respectively. Before that, we explain how to use these results to deduce the upper bound in Theorem \ref{Thm:Main}.

\begin{proof}[Proof of the upper bound in Theorem \ref{Thm:Main} assuming Theorems \ref{Thm: smooth} and \ref{Thm:TailPolya}]
Let $Z=Q^{21/40}$. Then, by \eqref{Eq:MV}, we have
\begin{align}\label{Eq:ThmMainupper1}
     M(\chi)\leq \dfrac{1}{2\pi}\max_{\alpha\in [0,1)}\bigg|\sum_{\substack{1\leq |n|\leq Z\\P^+(n)\leq y}}\dfrac{\chi(n)e(n\alpha)}{n}\bigg| + \dfrac{1}{\pi}\max_{\alpha\in [0,1)}\bigg|\sum_{\substack{1\leq n\leq Z\\P^+(n)> y}}\dfrac{\chi(n)e(n\alpha)}{n}\bigg| + C_2,
\end{align}
for some positive constant $C_2$. By Theorem \ref{Thm: smooth}, for any $\ep>0$ we have
\begin{align}\label{Eq:ThmMainupper2}
    \dfrac{1}{2\pi}\max_{\alpha\in [0,1)}\bigg|\sum_{\substack{1\leq |n|\leq Z\\P^+(n)\leq y}}\dfrac{\chi(n)e(n\alpha)}{n}\bigg|\leq (\log y)^\beta (\log_2 y)^{-1/4+\ep},
\end{align}
if $y$ is sufficiently large. Let $\ep>0$ be small and put
\begin{align*}
    y=\exp\big((V-C_3)^{1/\beta}(\log V)^{1/(4\beta)-\ep}\big),
\end{align*}
where $C_3$ will be chosen appropriately.
With the above choice of $y$, we combine \eqref{Eq:ThmMainupper1} and \eqref{Eq:ThmMainupper2} to obtain
\begin{align*}
    M(\chi)\leq V- C_3 + \dfrac{1}{\pi}\max_{\alpha\in [0,1)}\bigg|\sum_{\substack{1\leq n\leq Z\\P^+(n)> y}}\dfrac{\chi(n)e(n\alpha)}{n}\bigg| + C_2.
\end{align*}
We choose $C_3=C_2+1/\pi$. Therefore, the proportion of $\chi\in \F$ such that $M(\chi)>V$ is bounded above by the proportion of $\chi\in \mathcal{F}_3(Q)$ such that
\begin{align*}
    \max_{\alpha\in [0, 1]}\bigg|\sum_{\substack{1\leq n\leq Z\\P^+(n)> y}}\dfrac{\chi(n)e(n\alpha)}{n}\bigg|>1.
\end{align*}
We apply Theorem \ref{Thm:TailPolya} with $h(n)=1$ for all $n\in \mathbb{N}$ to deduce that
\begin{align*}
     \frac{1}{\#\F}\#\{\chi \in \F\colon M(\chi)>V\} \leq \exp\big(-\exp\left(V^{1/\beta}(\log V)^{1/(4\beta)+o(1)}\right)\big),
\end{align*}
as desired.
\end{proof}

Let us now turn to the lower bound in Theorem \ref{Thm:Main}. Our key strategy is to relate $M(\chi)$ to the values of certain associated Dirichlet $L$-functions at $1$. To make this precise, let $q\leq Q$ be large and $m\leq Q/(\log Q)^2$. Let $\chi$ be a primitive cubic character modulo $q$ and $\psi$ be an odd primitive character modulo $m$. Then, by \cite[Proposition 5.1]{LaMa22}, we have
\begin{align}\label{Eq:Mchi and L1}
    M(\chi)\gg \dfrac{\sqrt{m}}{\varphi(m)}|L(1, \chi \overline{\psi})|.
\end{align}
We will choose $m$ to be a non-exceptional modulus. To this end, we define this notion below.

\begin{defin}\label{Def: exceptional}
We say that an integer $m\geq 1$ is an exceptional modulus if there exists a Dirichlet character $\chi_m$ and a complex number $s$ such that $L(s, \chi_m)=0$ and
\begin{align*}
    \re(s)\geq 1-\dfrac{c}{\log (m(\textup{Im}(s)+2))}
\end{align*}
for some sufficiently small constant $c>0$.
\end{defin}
Given the above definition, the lower bound in Theorem \ref{Thm:Main} is a direct consequence of the inequality \eqref{Eq:Mchi and L1} and the following result.

\begin{thm}\label{Thm: L(1, chipsi)}
Let $Q$ be large and 
\[C_4\leq V\leq (\log_2 Q)^\beta (\log_4 Q)^{-C_5},\]
where $\beta$ is given by \eqref{Def:beta} and $C_4, C_5$ are suitably large constants.    Let $m$ be a non-exceptional modulus which is prime and verifies $m\in [\sqrt{\log V}, 2\sqrt{\log V}]$. Let $\psi$ be a primitive Dirichlet character modulo $m$ of order $m-1$.  Then, we have
\begin{align*}
        \dfrac{1}{\# \mathcal{F}_3(Q)}\#\bigg\{\chi\in \mathcal{F}_3(Q)\colon 
        |L(1, \chi\overline{\psi})|>V\bigg\}= \exp\left(-\exp\left(V^{1/\beta}(\log_2 V)^{O(1)}\right)\right). 
    \end{align*}
\end{thm}

We will prove Theorem \ref{Thm: L(1, chipsi)} in Section \ref{Sec: lower}. We quickly demonstrate how to obtain the lower bound in Theorem \ref{Thm:Main} using the above result.

\begin{proof}[Proof of the lower bound in Theorem \ref{Thm:Main} assuming Theorem \ref{Thm: L(1, chipsi)}]
Let $m$ be a non-exceptional modulus which is prime and verifies $m\in [\sqrt{\log V}, 2\sqrt{\log V}]$, and $\psi$ be a primitive Dirichlet character modulo $m$ of order $m-1$. First by  \eqref{Eq:Mchi and L1}, we have
\begin{equation}\label{Eq:LowerBoundML1}
\dfrac{1}{\#\F}\#\{\chi\in \F\colon M(\chi)>V\} \geq 
\dfrac{1}{\#\F}\#\{\chi\in \F\colon |L(1, \chi\overline{\psi})|>V_1\},
\end{equation}
    where $V_1=C_6V\varphi(m)/\sqrt{m}$ for some positive constant $C_6$. We now apply Theorem \ref{Thm: L(1, chipsi)} with $V$ replaced by $V_1=V(\log V)^{1/4} (\log_3V)^{O(1)}$ so that $$V_1\leq (\log_2 Q)^\beta (\log_4 Q)^{-C_7}$$ for some suitably large constant ${C_7}$. This gives 
$$
     \dfrac{1}{\#\F}\#\{\chi\in \F\colon |L(1, \chi\overline{\psi})|>V_1\}\geq \exp\left(-\exp\left(V^{1/\beta}(\log V)^{1/(4\beta)}(\log_2 V)^{O(1)}\right)\right).
 $$
Combining this estimate with \eqref{Eq:LowerBoundML1} completes the proof.
\end{proof}

\subsection{Organization of the paper} The remaining sections of the paper are organized as follows. We use Section \ref{Sec:Prelim} to gather several preliminary results on sums of divisor functions, moments of random multiplicative functions, and large sieve inequalities. We will establish Theorem \ref{Thm: smooth} in Section \ref{Sec: smooth part}. Next, we will devote Section \ref{Sec: tail} to the proof of Theorem \ref{Thm:TailPolya}. In Section \ref{Sec: lower}, we will prove Theorem \ref{Thm: L(1, chipsi)}. Finally, we will prove our structure result, Theorem \ref{Thm:Structure} in Section \ref{Sec: structure} by using ingredients from Sections \ref{Sec: smooth part} and \ref{Sec: tail}.


\section{Preliminary Results}\label{Sec:Prelim}

\subsection{Estimates on divisor sums}
In this section, we establish bounds on certain divisor sums. To this end, we shall express these sums as moments of certain weighted sums of a random completely multiplicative function of order $3$. Let  $\{\X(p)\}_{p \text{ prime}}$ be I. I. D. random variables taking the values $1,  \omega_3,  \omega_3^2$ with equal probability $1/3$. We extend $\X(n)$ multiplicatively to all positive integers by setting $\X(1)=1$ and $\X(n)= \prod_{p^{\ell} || n} \X(p)^{\ell}$.  Let $s >1/2$ be a real number. We consider the following random Dirichlet series 
$$ \sum_{n=1 }^{
\infty}\frac{\X(n)}{n^s}=\prod_{p\: \textup{prime}}\left(1-\frac{\X(p)}{p^s}\right)^{-1},$$
where both the series and the Euler product are almost surely convergent by Kolmogorov's two-series Theorem since $\ex(\X(p))=0$ for all primes $p$. Moreover we note that
\begin{equation}\label{Eq:OrthogonalityRandom}
\ex\left(\X(n)\overline{\X(m)}\right)=\ex\left(\X(nm^2)\right)=\begin{cases}1 & \textup{ if } nm^2 \textup{ is a cube,}\\
0 & \textup{ otherwise.}
\end{cases}
\end{equation} Hence, we obtain
$$ \ex\left(\left|\sum_{n=1}^{
\infty}\frac{\X(n)}{n^s}\right|^{2k}\right)= \ex\left(\sum_{n=1 }^{
\infty}\frac{\X(n)d_k(n)}{n^s}\sum_{m=1 }^{
\infty}\frac{\overline{\X(m)}d_k(m)}{m^s}\right)= \sum_{\substack{ n, m \geq 1\\ nm^2 \textup{ is a cube}}} \frac{d_k(n)d_k(m)}{(nm)^{s}}, $$
{}where $d_k(n)=\sum_{n_1\cdots n_k=n}1$ is the $k$-th divisor function. Therefore, we deduce that 
\begin{equation}\label{Eq:IdentityDivisorRandom}
\sum_{\substack{ n, m \geq 1\\ nm^2 \textup{ is a cube}}} \frac{d_k(n)d_k(m)}{(nm)^{s}}= \prod_{p}\ex\left(\left|1-\frac{\X(p)}{p^s}\right|^{-2k}\right)= \prod_{p}\ex\left(\left(1-\frac{2\re\X(p)}{p^s}+ \frac{1}{p^{2s}}\right)^{-k}\right),  
\end{equation}
by the independence of the $\X(p)$'s.
\begin{lem}\label{Lem:DivisorSums} Let $k$ be a large positive integer. Then for any $0\leq \eta \leq 1/\log k$ we have
$$
\sum_{\substack{ n, m \geq 1\\ nm^2 \textup{ is a cube}}} \frac{d_k(n)d_k(m)}{(nm)^{1-\eta}}=\exp\big(O(k\log\log k)\big).
$$
\end{lem}
\begin{proof}  By \eqref{Eq:IdentityDivisorRandom} we have 
$$\sum_{\substack{ n, m \geq 1\\ nm^2 \textup{ is a cube}}} \frac{d_k(n)d_k(m)}{(nm)^{1-\eta}}= \prod_{p}\ex\left(\left|1-\frac{\X(p)}{p^{1-\eta}}\right|^{-2k}\right). $$
If $p\leq 10k$ we  use the trivial inequality
$$
\ex\left(\left|1-\frac{\X(p)}{p^{1-\eta}}\right|^{-2k}\right) \leq \left(1-\frac{1}{p^{1-\eta}}\right)^{-2k} = \exp\left(O\left(\frac{k}{p}\right)\right), 
$$
while if $p> 10k
$ we use the estimate 
\begin{equation}\label{Eq:RandomLargeprimes}
\begin{aligned}
\ex\left(\left|1-\frac{\X(p)}{p^{1-\eta}}\right|^{-2k}\right)
&=\ex\left(\left(1-\frac{2\re\X(p)}{p^{1-\eta}}+ \frac{1}{p^{2-2\eta}}\right)^{-k}\right)\\
&=\ex\left(1+\frac{2k\re\X(p)}{p^{1-\eta}}+O\left(\frac{k^2}{p^{2-2\eta}}\right)\right)\\
& = 1+O\left(\frac{k^2}{p^{2-2\eta}}\right) =\exp\bigg(O\bigg(\dfrac{k^2}{p^{2-2\eta}}\bigg)\bigg),
\end{aligned}
\end{equation}where we have used the fact that $\ex (\X(p))=0$. Combining these estimates and using the Prime Number Theorem, we derive 
$$
\sum_{\substack{ n, m \geq 1\\ nm^2 \textup{ is a cube}}} \frac{d_k(n)d_k(m)}{(nm)^{1-\eta}}=\exp\left(O\left(k\sum_{p\leq 10 k}\frac{1}{p}+ k^2\sum_{p>10k}\frac{1}{p^{2-2\eta}}\right)\right)=\exp\big(O(k\log\log k)\big),
$$
as desired. 
\end{proof}
\begin{lem}\label{Lem:DivisorRoughBig1}
There exists a positive constant $C_8$ such that for any integer $k\ge 2$ and real number $y>10k$  we have 
 $$
\mathcal{D}_y(k):=\sum_{\substack{ n, m>1\\P^{-}(nm)>y\\ nm^2 \textup{ is a cube}}} \frac{d_k(n)d_k(m)}{nm}\ll \left(\frac{C_8k}{y\log y}\right)^{k}.
$$
\end{lem}
\begin{proof}
We start by proving a weaker bound for $\mathcal{D}_y(k)$ which will be useful in the proof. By \eqref{Eq:OrthogonalityRandom} and \eqref{Eq:RandomLargeprimes} we observe that
\begin{equation}\label{Eq:DivisorSumRough}\begin{aligned}
\mathcal{D}_y(k)&\leq \sum_{\substack{ P^{-}(nm)>y\\ nm^2 \textup{ is a cube}}} \frac{d_k(n)d_k(m)}{nm}  = \ex\left(\left|\sum_{P^{-}(n)>y}\frac{d_k(n)\X(n)}{n}\right|^2\right)\\&= \prod_{p>y} \ex\left(\left|1-\frac{\X(p)}{p}\right|^{-2k}\right)= \exp\left(O\left(k^2\sum_{p>y} \frac{1}{p^2}\right)\right)= \exp\left(O\left(\frac{k^2}{y\log y}\right)\right)
\end{aligned}
\end{equation}
by the Prime Number Theorem. Now using  \eqref{Eq:OrthogonalityRandom} we get 
\begin{equation}\label{Eq:SplitM}
\begin{aligned}\mathcal{D}_{y}(k)&= \ex\Bigg(\Bigg|\sum_{\substack{n>1\\P^{-}(n)>y}}\frac{\X(n)}{n}\Bigg|^{2k}\Bigg)\\
&=\ex\Bigg(\Bigg|\sum_{\substack{n>1\\ P^{-}(n)>y}}  \frac{\X(n)}{n}\Bigg|^{2k} \cdot \mathbf{1}_{\mathcal{A}}\Bigg)+\ex\Bigg(\Bigg|\sum_{\substack{n>1\\ P^{-}(n)>y}}  \frac{\X(n)}{n}\Bigg|^{2k}\cdot \mathbf{1}_{\mathcal{A}^c}\Bigg), 
\end{aligned}
\end{equation}
where $\mathbf{1}_{\mathcal{B}}$ is the indicator function of the event $\mathcal{B}$, and $\mathcal{A}$ is the event $$\mathcal{A} :=\big\{ |\mathbb{W}| \leq 1\big\}, \text{ where } \mathbb{W}:= \sum_{p>y} \frac{\X(p)}{p}. $$
Let $\ell$ be  a positive integer to be chosen
. By Markov's inequality and \eqref{Eq:OrthogonalityRandom}, we have 
$$
\pr(\mathcal{A}^c)\leq \ex(|\mathbb{W}|^{2\ell})=\sum_{\substack{p_1, \dots, p_{2\ell}>y \\ p_1\cdots p_{\ell}p_{\ell+1}^2\cdots p_{2\ell}^2 \textup{ is a cube}}}\frac{1}{p_1\cdots p_{2\ell}}.
$$
Furthermore, it follows from the proof of Lemma 6.5 of \cite{DDLL23} that 
\begin{equation}\label{Eq:MomentsRandomprimes}
\begin{aligned}
 \sum_{\substack{p_1, \dots, p_{2\ell}>y \\ p_1\cdots p_{\ell}p_{\ell+1}^2 \cdots p_{2\ell}^2 \textup{ is a cube}}}\frac{1}{p_1\cdots p_{2\ell}} \leq 2^{\ell}\ell!\left(\sum_{p>y} \frac{1}{p^2}\right)^{2\ell}\ll \left(\frac{c\ell}{y\log y}\right)^{\ell},
\end{aligned}
\end{equation}
for some positive constant $c$. 
Choosing $\ell=\lfloor y(\log y)/3c\rfloor$
we derive
\begin{equation}\label{Eq:BoundProbComplement}
\pr(\mathcal{A}^c) \ll \exp\left(-\frac{y\log y}{3c}\right).
\end{equation}
Therefore, using the Cauchy-Schwarz inequality and the bounds \eqref{Eq:DivisorSumRough} and \eqref{Eq:BoundProbComplement}, we deduce that the contribution of the second part in \eqref{Eq:SplitM} is 
\begin{equation}\label{Eq:ContrSecondPart}\ex\Bigg(\Bigg|\sum_{\substack{n>1\\ P^{-}(n)>y}} \frac{\X(n)}{n}\Bigg|^{2k}\cdot \mathbf{1}_{\mathcal{A}^c}\Bigg) \leq  \mathcal{D}_{y}(2k)^{1/2} \cdot \pr(\mathcal{A}^c)^{1/2} \ll \exp\left(-\frac{y\log y}{6c}+ O\left(\frac{k^2}{y\log y}\right)\right).
\end{equation}
We now estimate the contribution of the first part in \eqref{Eq:SplitM}. Note that on the event $\mathcal{A}$ we have  $|e^{\mathbb{W}}-1|\leq e |\mathbb{W}|$ and 
$$ 
\sum_{\substack{n>1\\ P^{-}(n)>y}}  \frac{\X(n)}{n} = -1+\prod_{p>y}\left(1-\frac{\X(p)}{p}\right)^{-1}= -1+ e^{\mathbb{W} +O(1/y\log y)}= e^{\mathbb{W}}-1+ O\left(\frac{1}{y\log y}\right).
$$
Therefore, using Minkowski's inequality and \eqref{Eq:MomentsRandomprimes} we derive 
\begin{align*}\ex\Bigg(\Bigg|\sum_{\substack{n>1\\ P^{-}(n)>y}}  \frac{\X(n)}{n}\Bigg|^{2k} \cdot \mathbf{1}_{\mathcal{A}}\Bigg)^{1/2k}
&\le e \ex(|\mathbb{W}|^{2k})^{1/2k}+ O\left(\frac{1}{y\log y}\right)\ll \sqrt{\frac{k}{y\log y}}.
\end{align*}
Combining the above estimate with \eqref{Eq:SplitM} and \eqref{Eq:ContrSecondPart} completes the proof.
\end{proof}

\subsection{The large sieve}  We will need the following two large sieve type inequalities. The first is the classical large sieve for the family of all primitive characters $\chi \pmod q$ with $q\leq Q$ (see for example \cite{Ga67}). 
The second is a consequence of cubic reciprocity and the P\'olya-Vinogradov inequality for Hecke characters over $\mathbb{Q}(\omega_3)$.

\begin{lem}\label{Lem:LargeSieve}
Let $\{a_m\}_{m\geq 1}$ be an arbitrary sequence of complex numbers. Then we have 
$$\sum_{q\leq Q} \ \sums_{\chi\bmod q} \left| \sum_{m\leq M} a_m \chi(m)\right|^2 \ll  (Q^2+M) \sum_{m\leq M} |a_m|^2,$$
where $\sums$ indicates that the summation is over primitive characters. 
\end{lem}
\begin{lem}\label{Lem:Orthogonality}
Let $\{a_m\}_{m\geq 1}$ be an arbitrary sequence of complex numbers. Then we have 
\begin{align*}
&\sum_{\chi \in  \F} \left| \sum_{m\leq M} a_m \chi(m)\right|^2 \\
& \quad \quad \quad \ll  Q \sum_{\substack{m_1, m_2\leq M\\ m_1m_2^2 \textup{ is a cube}}} |a_{m_1}a_{m_2}| + O\left(\sqrt{Q}(\log Q)M(\log M)^{3/2}\left(\sum_{m\leq M}|a_m|\right)^2\right) .
\end{align*}
\end{lem}

\begin{proof}
By expanding the square we find that 
\begin{equation}\label{Eq:Orthogonality1}
\sum_{\chi \in  \F} \left| \sum_{m\leq M} a_m \chi(m)\right|^2= \sum_{m_1, m_2\leq M} a_{m_1}\overline{a_{m_2}}\sum_{\chi\in \F} \chi(m_1m_2^2),
\end{equation}
since $\overline{\chi(m_2)}=\chi(m_2)^2.$ The contribution of the diagonal terms $m_1, m_2$ for which $m_1m_2^2$ is a cube is 
$$\ll Q\sum_{\substack{m_1, m_2\leq M\\ m_1m_2^2 \textup{ is a cube}}} |a_{m_1}a_{m_2}|, $$
since $|\F|\ll Q.$ Now if $m_1m_2^2$ is not a cube, we use Lemma 3.5 of \cite{DDLL23},
which implies that
$$\sum_{\chi\in \F} \chi(m_1m_2^2)\ll \sqrt{Q}(\log Q)(m_1m_2)^{1/2}\log(m_1m_2^2)^{3/2}.$$
Therefore, the contribution of the off-diagonal terms to the left hand side of \eqref{Eq:Orthogonality1} is 
$$ \ll \sqrt{Q}(\log Q)M(\log M)^{3/2}\left(\sum_{m\leq M}|a_m|\right)^2,$$
as desired.
 \end{proof}



\section{The smooth part in P\'olya's Fourier expansion: Proof of Theorem \ref{Thm: smooth}}\label{Sec: smooth part}

Our proof of Theorem \ref{Thm: smooth} 
builds on the recent work of Lamzouri and Mangerel \cite{LaMa22}. Before that, we fix a few notations. Let $y\geq 1$. For any two completely multiplicative functions $f$ and $g$ with $|f(n)|, |g(n)|\leq 1$ for all $n\geq 1$, we define
\begin{align}\label{Def: Distance}
    \mathbb{D}(f, g; y):=\bigg(\sum_{p\leq y}\dfrac{1-\re(f(p)\overline{g(p)})}{p}\bigg)^{1/2}.
\end{align}
We say that $f$ \emph{pretends} to be $g$ (up to $y$) if there is a constant $\delta \in [0,1)$ such that
\[\mathbb{D}(f, g; y)^2\leq \delta\log_2y.\]
Note that by Mertens' theorem, we have
\begin{align*}
    \mathbb{D}(f, g; y)^2\leq 2\log_2y + O(1).
\end{align*}
Next, for any $x\geq 2$ and $T>0$, we define
\begin{align}\label{Def: Mdistance}
    \mathcal{M}(f; x, T):=\min_{|t|\leq T}\mathbb{D}(f, n^{it}; x)^2.
\end{align}

We begin with the following result of Lamzouri and Mangerel  \cite{LaMa22}.

\begin{lem}[Major arc estimates]\label{Lem: major1}
Let $N\geq 2$ and $y\geq 100$. Let $1\leq r\leq (\log y)^{4/11}$ be an integer and $(b, r)=1$. Let $\chi\pmod q$ be an even primitive character. Suppose that as $\psi$ ranges over all primitive characters of conductor less than $r$, the minimum of $\mathcal{M}(\chi\overline{\psi}; y, (\log y)^{-7/11})$ is attained when $\psi=\xi\pmod m$. Then, we have
\begin{align*}
&\sum_{\substack{1\leq |n|\leq N\\ P^+(n)\leq y}}\dfrac{\chi(n)e(nb/r)}{n}\\
&  \ll \begin{cases} \displaystyle{\dfrac{m^{3/2}}{\varphi(m) r}\prod_{p\mid \frac{r}{m}}\left(1+\frac{p+1}{p-1}\right)(\log y)e^{-\mathcal{M}(\chi\overline{\xi}; y, (\log y)^{-{7}/{11}})} + (\log y)^{{9}/{11}+o(1)}} & \textup{ if } \xi \textup{ is odd and } m\mid r,\\
    (\log y)^{9/11+o(1)} & \textup{ otherwise}.\end{cases}
\end{align*}
In particular, in the first case, we have
$$\sum_{\substack{1\leq |n|\leq N\\ P^+(n)\leq y}}\dfrac{\chi(n)e(nb/r)}{n}\ll \frac{{m}^{1/2}}{\varphi(m)}(\log y)e^{-\mathcal{M}(\chi\overline{\xi}; y, (\log y)^{-{7}/{11}})} + (\log y)^{9/11+o(1)}.$$
\end{lem}

\begin{proof} 
This follows from the proof of Theorem 2.3 of \cite{LaMa22} (see \cite[Eq (4.3)]{LaMa22} onwards). 
    \end{proof}

    To apply the above lemma efficiently, we need a reasonable lower bound for the quantity $\mathcal{M}(\chi\overline{\psi}; y, (\log y)^{-7/11})$, which is given in the following lemma.

\begin{lem}\label{Lem: major2}
    Let $y\geq 100$, $\alpha\in (0,1)$ and $\varepsilon>0$ be small. Let $\chi\pmod q$ be a primitive cubic character. Let $\psi$ be an odd primitive character modulo $m$, with $m\leq (\log y)^{4\alpha/7}$, and let $k$ be the order of $\psi$. Then, 
    \begin{align*}
        \mathcal{M}(\chi\overline{\psi}; y, (\log y)^{-\alpha})\geq \bigg(1-\beta &+ \dfrac{\alpha\pi^2\beta}{36k^2}\bigg)\log_2 y  + O_{\alpha}(\log_2m)\\
        &- \begin{cases}
        \varepsilon\log m & \text{if $m$ is an exceptional modulus}\\
        0, &\text{otherwise}.
    \end{cases}
    \end{align*} 
\end{lem}
\begin{proof}
This is a consequence of \cite[Proposition 2.1]{LaMa22} by taking $g=3$.
\end{proof}

Let us now demonstrate how to complete the proof of Theorem \ref{Thm: smooth} using the above results.

\begin{proof}[Proof of Theorem \ref{Thm: smooth}]Recall that if $\chi\in \mathcal{F}_3(Q)$, $Z=Q^{21/40}$ and $y\leq \log Q$, we wish to show
\begin{align*}
    \max_{\alpha\in [0,1)}\bigg|\sum_{\substack{1\leq |n|\leq Z\\ P^+(n)\leq y}}\dfrac{\chi(n)e(n\alpha)}{n}\bigg|\ll (\log y)^{\beta}(\log_2y)^{-1/4+o(1)},
\end{align*}
where $\beta$ is given by \eqref{Def:beta}.

Let $\alpha \in [0,1)$. We apply Dirichlet's approximation theorem to obtain a rational approximation
\[\bigg|\alpha-\dfrac{b}{r}\bigg|\leq \dfrac{1}{rR},\]
where $1\leq r\leq R:=(\log y)^5$ and $(b, r)=1$. We will consider two cases, depending on whether $r$ is large or small. In order to make it precise, we set $M=(\log y)^{4/11}$. 

First, we deal with the minor arcs case, that is, when $M<r\leq R$. By \cite[Corollary 2.2]{Gol}, we have
\begin{align*}
   \bigg|\sum_{\substack{1\leq |n|\leq Z\\ P^+(n)\leq y}}\dfrac{\chi(n)e(n\alpha)}{n}\bigg|\ll \log R +  \dfrac{(\log M)^{5/2}}{M^{1/2}}(\log y) + \log_2y\ll (\log y)^{9/11}(\log_2 y)^{5/2},
\end{align*}
which is acceptable since $\beta>9/11$.

Let us now deal with the major arcs case, that is, when $1\leq r\leq M$. Set $N=\min\{Z, |r\alpha-b|^{-1}\}$. Then, we apply \cite[Lemma 4.1]{Gol} to obtain
\begin{align*}
    \sum_{\substack{1\leq |n|\leq Z\\ P^+(n)\leq y}}\dfrac{\chi(n)e(n\alpha)}{n}=\sum_{\substack{1\leq |n|\leq N\\ P^+(n)\leq y}}\dfrac{\chi(n)e(nb/r)}{n} + O\bigg(\dfrac{(\log R)^{3/2}}{R^{1/2}}(\log y)^2 + \log R + \log_2y\bigg).
\end{align*}
Since $R=(\log y)^5$, the contribution of the big Oh term is $\ll \log_2y$. This implies that
\begin{align*}
    \sum_{\substack{1\leq |n|\leq Z\\ P^+(n)\leq y}}\dfrac{\chi(n)e(n\alpha)}{n}=\sum_{\substack{1\leq |n|\leq N\\ P^+(n)\leq y}}\dfrac{\chi(n)e(nb/r)}{n} + O(\log_2y).
\end{align*}
So, the goal now reduces to showing that
\begin{align}\label{Eqn:smooth1}
    \sum_{\substack{1\leq |n|\leq N\\ P^+(n)\leq y}}\dfrac{\chi(n)e(nb/r)}{n} \ll (\log y)^{\beta}(\log_2y)^{-1/4+o(1)}.
\end{align}
Let $\xi$ be the character of conductor $m\leq (\log y)^{4/11}$ that minimizes $\mathcal{M}(\chi\overline{\psi}; y, (\log y)^{-7/11})$. We consider two cases depending on whether $\xi$ is even or odd.

 If $\xi$ is even, then by Lemma \ref{Lem: major1}, we have
\begin{align*}
    \sum_{\substack{1\leq |n|\leq N\\ P^+(n)\leq y}}\dfrac{\chi(n)e(nb/r)}{n} \ll (\log y)^{9/11 + o(1)} \ll (\log y)^{\beta-\ep}.
\end{align*}

On the other hand, if $\xi$ is odd, we will apply Lemmas \ref{Lem: major1} and \ref{Lem: major2}. To do that, let us write $k$ to denote the order of $\xi$. For brevity, let us write $\eta(m)=1$ if $m$ is an exceptional modulus and $\eta(m)=0$ otherwise. Then, for some explicit constant $c_1>0$, we have
\begin{align*}
   \sum_{\substack{1\leq |n|\leq N\\ P^+(n)\leq y}}\dfrac{\chi(n)e(nb/r)}{n}  \ll \dfrac{{m}^{1/2}}{\varphi(m)}(\log y)^\beta \exp\bigg(-\dfrac{c_1}{k^2}\log_2y+\eta(m)\varepsilon\log m + O(\log_2m)\bigg).
\end{align*}
Since $1/\varphi(m)\ll (\log_2 m)/m$  we get
\begin{align*}
    \sum_{\substack{1\leq |n|\leq N\\ P^+(n)\leq y}}\dfrac{\chi(n)e(nb/r)}{n}\ll (\log y)^\beta \exp\bigg(-\bigg(\dfrac{1}{2}-\eta(m)\varepsilon\bigg)\log m -\dfrac{c_1}{m^2}\log_2y + c_2\log_2m\bigg),
\end{align*}
for some positive constant $c_2$. 
Optimizing the expression inside the exponential in terms of $m$ shows that the maximum is attained when  $m\asymp (\log _2 y)^{1/2}$. This implies 
\begin{align*}
    \sum_{\substack{1\leq |n|\leq N\\ P^+(n)\leq y}}\dfrac{\chi(n)e(nb/r)}{n}\ll (\log y)^\beta (\log_2y)^{-1/4 + o(1)},
\end{align*}
as desired.
\end{proof}

\section{The distribution of the tail in P\'olya's Fourier expansion: Proof of Theorem \ref{Thm:TailPolya}} \label{Sec: tail}
Let $\{a_n\}_{n\geq 1}$ be an arbitrary sequence of complex numbers such that $|a_n|\leq 1$ for all integers $n\geq 1$. We start by proving the following elementary lemma, which is a generalization of  \cite[Lemma 4.4] {La22}. This will be used to reduce the problem of bounding the moments and the distribution of  the quantity 
\begin{equation}\label{Eq:TailPolyaTwisted}
\max_{t\in [0,1]}\max_{\alpha\in [0, 1)}\Bigg|\sum_{1\leq n\leq Z} \frac{\chi(n) a_n e(n\alpha)}{n^{1+it}}\Bigg|
\end{equation}
to a setting where we can apply the large sieve inequalities in Lemmas \ref{Lem:LargeSieve} and \ref{Lem:Orthogonality}.
\begin{lem}\label{Lem:Reduction}
Let $\mathcal{X}$ be a set of Dirichlet characters, and $2\leq N_1 <N_2\leq R$ be real numbers. Let $\{a_n\}_{n\geq 1}$ be an arbitrary sequence of complex numbers such that $|a_n|\leq 1$ for all $n$. Define  
$\mathcal{A}:= \{b/R : 1\leq b\leq R\},$ and $\mathcal{T}:= \{\ell/T  : 0\leq \ell\leq T\}$, where $T=R(\log N_2)^2/N_2$. Then, for any positive integer $k\geq 1$, we have 
\begin{align*}
&\Bigg(\sum_{\chi\in \mathcal{X}} \max_{t\in [0, 1]}\max_{\alpha\in [0,1)} \Bigg|\sum_{N_1\leq  n\leq N_2} \frac{\chi(n) a_n e(n\alpha)}{n^{1+it}}\Bigg|^{2k}\Bigg)^{1/2k} \\
& \leq \Bigg(\sum_{t\in \mathcal{T}}\sum_{\alpha \in \mathcal{A}}  \sum_{\chi\in \mathcal{X}} \Bigg|\sum_{N_1^k\leq n\leq N_2^k}\frac{\chi(n)g_{N_1, N_2, k}(n, \alpha)}{n^{1+it}}\Bigg|^{2}\Bigg)^{1/2k}+ O\left(
|\mathcal{X}|^{1/2k} N_2R^{-1}\right),
\end{align*}
where 
\begin{equation}\label{Defgn}
g_{N_1, N_2, k}(n, \alpha):= \sum_{\substack{N_1\leq n_1, \dots, n_k \leq N_2\\ n_1\cdots n_k= n}}\prod_{j=1}^k a_{n_j}e(n_j\alpha),
\end{equation}
and the implicit constant in the error term is absolute. We also note that \begin{equation}\label{Eq:Bounbdgn}
|g_{N_1, N_2, k}(n, \alpha)|\leq d_k(n),   
\end{equation} for all positive integers $k$ and $n$, uniformly in $N_1,N_2$, and $\alpha.$
\end{lem}
\begin{proof}
Let $\alpha\in [0, 1)$ and $t\in [0, 1]$. Then there exist $\lambda_{\alpha}\in \mathcal{A}$ and $s_t\in \T$ such that $|\alpha-\lambda_{\alpha}|\leq 1/R$ and $|t-s_t|\leq 1/T$.
In this case we have $e(n\alpha)= e(n\lambda_{\alpha})+ O(n/R),$
 and $n^{-1-it}= n^{-1-is_t}+O((\log n)/(nT)).$ Therefore, we deduce that
$$\max_{t\in [0, 1]}\max_{\alpha\in [0,1)} \Bigg|\sum_{N_1\leq n\leq N_2} \frac{\chi(n)a_n e(n\alpha)}{n^{1+it}}\Bigg| = \max_{s\in \T}\max_{\lambda\in \mathcal{A}} \Bigg|\sum_{N_1\leq n\leq N_2} \frac{\chi(n) a_ne(n\lambda)}{n^{1+is}}\Bigg| +O\left(\frac{N_2}{R}\right). 
 $$
We now use Minkowski's inequality, which implies 
\begin{align*}
 &\Bigg(\sum_{\chi \in \mathcal{X}} \max_{t\in [0,1]}\max_{\alpha\in [0,1)} \Bigg|\sum_{N_1\leq n\leq N_2} \frac{\chi(n) a_ne(n\alpha)}{n^{1+it}}\Bigg|^{2k}\Bigg)^{1/2k}\\
 & \leq  \Bigg(\sum_{\chi \in \mathcal{X}} \max_{t\in \T}\max_{\alpha \in \mathcal{A}} \Bigg|\sum_{N_1\leq n\leq N_2} \frac{\chi(n) a_ne(n\alpha)}{n^{1+it}}\Bigg|^{2k}\Bigg)^{1/2k}+ O\left(|\mathcal{X}|^{1/2k} N_2R^{-1}\right)\\
 & \leq  \Bigg(\sum_{t\in \T}\sum_{\alpha \in \mathcal{A}}\sum_{\chi \in \mathcal{X}}  \Bigg|\sum_{N_1\leq  n\leq N_2} \frac{\chi(n)a_n e(n\alpha)}{n^{1+it}}\Bigg|^{2k}\Bigg)^{1/2k}+ O\left(|\mathcal{X}|^{1/2k} N_2R^{-1}\right).\\
\end{align*}
The lemma follows upon noting that 
\[\Bigg|\sum_{N_1\leq n\leq N_2}\frac{\chi(n) a_ne(n\alpha)}{n^{1+it}}\Bigg|^{2k}= \Bigg|\sum_{N_1^k\leq n\leq N_2^k} \frac{\chi(n)g_{N_1, N_2, k}(n, \alpha)}{n^{1+it}}\Bigg|^{2}.\]
\end{proof}
Using the above result together with the  classical large sieve inequality in Lemma \ref{Lem:LargeSieve}, we prove that the part of the sum \eqref{Eq:TailPolyaTwisted} corresponding to the range  $Q^{\delta}\leq n\leq Q^{21/40}$ is very small for all cubic characters $\chi\in \F$, except for a set of size $O(Q^{1-\delta})$, for some small $\delta>0.$

\begin{lem}\label{Lem:TailLarge} 
Let  $\delta=1/1000$. Let $\{a_n\}_{n\geq 1}$ be an arbitrary sequence of complex numbers such that $|a_n|\leq 1$ for all $n$.  
There are at most $O(Q^{1-\delta})$ cubic characters $\chi\in \F$ such that 
\begin{equation}\label{Eq:ConditionLargeTail}
\max_{t\in [0,1]}\max_{\alpha\in [0, 1)}\left|\sum_{Q^{\delta}\leq n\leq Q^{21/40}} \frac{\chi(n) a_ne(n\alpha)}{n^{1+it}}\right|\geq \frac{1}{\log Q}.
\end{equation}
\end{lem}

\begin{proof} 
As before we set $Z=Q^{21/40}.$ We shall split the range of summation into dyadic parts. Let $L_1:= \lfloor \log Q/(1000\log 2)\rfloor$ and  $L_2:= \lfloor \log Z/\log 2\rfloor$. We put $s_{L_1}:= Q^{\delta}$, $s_{L_2+1}:= Z$, and we define $s_{\ell}:=2^{\ell}$ for $L_1+1\leq \ell\leq L_2$.
Let 
$$
k=k_{\ell}:=\begin{cases} 4 & \textup{ if } \sqrt{Q}\leq s_{\ell}\leq Z, \\ 
\lfloor 2\log Q/\log s_{\ell}\rfloor +1 & \textup{ if } Q^{\delta} \leq s_{\ell}< \sqrt{Q}.
\end{cases} 
$$
We observe that $4\leq k\leq 2001$ by our assumption on $s_{\ell}$, and that $s_{\ell}^k\geq Q^2$ for all $\ell$. Let $L_1\leq \ell\leq L_2$.  Using Lemma \ref{Lem:Reduction} with the choice $R=s_{\ell}^{1+\delta}$ together with the elementary inequality \begin{equation}\label{Eq:ElementaryInequality}
|a+b|^{2k}\leq 2^{2k}(|a|^{2k}+|b|^{2k})
\end{equation}
valid for all $a, b\in \mathbb{C}$, we obtain 
\begin{equation}\label{eqTail2}
\begin{aligned}
&\sum_{\chi \in \F} \max_{t\in [0,1]}\max_{\alpha\in [0,1)} \left|\sum_{s_{\ell}\leq n\leq s_{\ell+1}} \frac{\chi(n) a_ne(n\alpha)}{n^{1+it}}\right|^{2k} \\
& \ll \sum_{t\in \T}\sum_{\alpha \in \mathcal{A}}  \sum_{\chi \in \F}\left|\sum_{s_{\ell}^k\leq n\leq s_{\ell+1}^k} \frac{\chi(n)g_{s_{\ell}, s_{\ell+1}, k}(n, \alpha)}{n^{1+it}}\right|^{2}+ \frac{Q}{s_{\ell}^{2k\delta}}.\\
\end{aligned}
\end{equation}
Let $\eta=1/(10k)$. Using Lemma \ref{Lem:LargeSieve} together with the bound $d_k(n)\ll n^{\eta}$, which is valid since $k$ is bounded,  we obtain
\begin{equation}\label{eqTail20}
\begin{aligned}
\sum_{\chi\in \F}  \left|\sum_{s_{\ell}^k\leq n\leq s_{\ell+1}^k} \frac{\chi(n)g_{s_{\ell}, s_{\ell+1}, k}(n, \alpha)}{n^{1+it}}\right|^{2}&\leq \sum_{q\leq Q} \ \sums_{\chi\bmod q}  \left|\sum_{s_{\ell}^k\leq n\leq s_{\ell+1}^k} \frac{\chi(n)g_{s_{\ell}, s_{\ell+1}, k}(n, \alpha)}{n^{1+it}}\right|^{2}\\
& \ll s_{\ell}^{k} \sum_{s_{\ell}^k\leq n\leq s_{\ell+1}^k}\frac{|g_{s_{\ell},s_{\ell+1}, k}(n, \alpha)|^2}{n^2}\\
& \ll s_{\ell}^{k} \sum_{n> s_{\ell}^k}\frac{1}{n^{2-2\eta}} \ll s_{\ell}^{1/5},
\end{aligned}
\end{equation}
using that $|g_{s_{\ell},s_{\ell+1}, k} (n, \alpha)|^2 \leq d_k(n)^2\ll_{\eta} n^{2\eta}$ by \eqref{Eq:Bounbdgn}. 
Inserting this estimate in \eqref{eqTail2} we obtain
$$ 
\sum_{\chi\in \F}   \max_{t\in [0,1]}\max_{\alpha\in [0,1)} \left|\sum_{s_{\ell}\leq n\leq s_{\ell+1}} \frac{\chi(n)a_n e(n\alpha)}{n^{1+it}}\right|^{2k} \ll s_{\ell}^{6/5+3\delta} + \frac{Q}{s_{\ell}^{2k\delta}} \ll Q^{1-4\delta},
$$
since $s_{\ell}^{2k}\geq Q^{4}.$ 
Thus, we deduce that the number of cubic characters $\chi\in \F$  such that 
$$\max_{t\in [0,1]}\max_{\alpha\in [0,1)} \left|\sum_{s_{\ell}\leq n\leq s_{\ell+1}} \frac{\chi(n)a_n e(n\alpha)}{n^{1+it}}\right|\geq \frac{1}{(\log Q)^2}$$ 
is 
\begin{equation}\label{Eq:BoundNumberChiLarge}
\ll (\log Q)^{4k} \sum_{\chi\in \F} \max_{t\in [0,1]}\max_{\alpha\in [0,1)} \left|\sum_{s_{\ell}\leq n\leq s_{\ell+1}} \frac{\chi(n)a_ne(n\alpha)}{n^{1+it}}\right|^{2k} \ll Q^{1-2\delta}. 
\end{equation} 
We now observe that 
$$\max_{t\in [0,1]}\max_{\alpha\in [0,1)} \left|\sum_{Q^{\delta}\leq n\leq Z} \frac{\chi(n)a_n e(n\alpha)}{n^{1+it}}\right|\leq \sum_{L_1\leq \ell\leq L_2}\max_{t\in [0,1]}\max_{\alpha\in [0,1)} \left|\sum_{s_{\ell}\leq n\leq s_{\ell+1}} \frac{\chi(n)a_n e(n\alpha)}{n^{1+it}}\right|.$$
Therefore, if $\chi$ verifies \eqref{Eq:ConditionLargeTail} then there exists $\ell \in [L_1, L_2]$ such that 
$$\max_{t\in [0,1]}\max_{\alpha\in [0,1)} \left|\sum_{s_{\ell}\leq n\leq s_{\ell+1}} \frac{\chi(n)a_n e(n\alpha)}{n^{1+it}}\right|\geq \frac{1}{L_2\log Q} \geq \frac{1}{(\log Q)^2}.$$
Hence, it follows from \eqref{Eq:BoundNumberChiLarge} that the number of characters $\chi\in \F$ which verify $\eqref{Eq:ConditionLargeTail}$ is $\ll L_2 Q^{1-2\delta}\ll Q^{1-\delta}$, as desired.
\end{proof}
Next, we use the cubic large sieve of Lemma \ref{Lem:Orthogonality} to bound the large moments of a dyadic part of the sum \eqref{Eq:TailPolyaTwisted} in the range $2\leq N\leq Q^{\delta}$. In particular, this will be sufficient to show that the contribution of the intermediate range $\exp((\log_2 Q)^2)\leq n\leq Q^{\delta}$ in \eqref{Eq:TailPolyaTwisted} is small for all cubic characters $\chi\in\F$, except for a small exceptional set. 
\begin{lem}\label{Lem:TailMedium} Let $\delta=1/1000$. Let $\{a_n\}_{n\geq 1}$ be an arbitrary sequence of complex numbers such that $|a_n|\leq 1$ for all $n$.  Let $N_1, N_2$ be two real numbers such that  $2\leq N_1\leq N_2 \leq 2N_1\leq Q^{\delta}.$ Then, for all positive integers $100\leq k
\leq \log Q/(10 \log N_1)$ we have 
$$ 
\sum_{\chi\in \F} \max_{t\in [0, 1]}\max_{\alpha\in [0,1)} \left|\sum_{N_1\leq n\leq N_2} \frac{\chi(n)a_n e(n\alpha)}{n^{1+it}}\right|^{2k} 
 \ll Q N_1^4\exp\left(-\frac{2k\log N_1}{\log k} +O(k\log_2 k )\right).
$$
\end{lem}
\begin{proof}
Using \eqref{Eq:ElementaryInequality}, Lemma \ref{Lem:Reduction} with $R=N_1^2$ together with Lemma \ref{Lem:Orthogonality} we obtain
\begin{equation}\label{Eq:Tail3}
\begin{aligned}
&\sum_{\chi\in \F} \max_{t\in [0,1]}\max_{\alpha\in [0,1)} \left|\sum_{N_1\leq n\leq N_2} \frac{\chi(n)a_n e(n\alpha)}{n^{1+it}}\right|^{2k} \\
& \ll e^{O(k)}\sum_{t\in \T}\sum_{\alpha \in \mathcal{A}}  \sum_{\chi\in \F} \left|\sum_{N_1^k\leq n\leq N_2^k} \frac{\chi(n)g_{N_1, N_2, k}(n, \alpha)}{n^{1+it}}\right|^{2}+ \frac{Qe^{O(k)}}{N_1^{2k}}\\
& \ll QN_1^2 e^{O(k)} \sum_{\alpha \in \mathcal{A}}\sum_{\substack{N_1^k\leq n_1, n_2 \leq N_2^k\\ n_1n_2^2 \textup{ is a cube}}}\frac{|g_{N_1, N_2, k}(n_1, \alpha)g_{N_1, N_2, k}(n_2, \alpha)|}{n_1n_2}+ E_1\\
& \ll Q N_1^{4} e^{O(k)}  \sum_{\substack{ n_1, n_2 \geq N_1^k\\ n_1n_2^2 \textup{ is a cube}}} \frac{d_k(n_1)d_k(n_2)}{n_1n_2}+ E_1,
\end{aligned}
\end{equation}
where 
$$ E_1 \ll e^{O(k)}\frac{Q}{N_1^{2k}} + e^{O(k)} N_1^4 Q^{3/4}\left(\sum_{N_1^k\leq n\leq N_2^k}\frac{|g_{N_1,N_2, k}(n, \alpha)|}{n}\right)^2\ll e^{O(k)}\frac{Q}{N_1^{2k}},$$
since $N_1^k\ll Q^{1/10}$ by our assumption on $k$ and
$$ \sum_{N_1^k\leq n\leq N_2^k}\frac{|g_{N_1,N_2, k}(n, \alpha)|}{n} \leq \left(\sum_{N_1\leq n\leq N_2}\frac1n\right)^k=e^{O(k)}.$$ 
 To bound the sum over $n_1, n_2$, on the right hand side of \eqref{Eq:Tail3} we shall use Rankin's trick. Choosing $\nu=1/\log k$ and using Lemma \ref{Lem:DivisorSums} we obtain
\begin{align*}
    \sum_{\substack{ n_1, n_2 \geq N_1^k\\ n_1n_2^2 \textup{ is a cube}}} \frac{d_k(n_1)d_k(n_2)}{n_1n_2} 
    &\leq  N_1^{-2k\nu}  \sum_{\substack{ n_1, n_2 \geq 1\\ n_1n_2^2 \textup{ is a cube}}} \frac{d_k(n_1)d_k(n_2)}{(n_1n_2)^{1-\nu}}\\
    &\ll \exp\left(-\frac{2k\log N_1}{\log k} +O(k\log_2 k)\right).
\end{align*}
Inserting this estimate in \eqref{Eq:Tail3} completes the proof.  
\end{proof}
Our last ingredient of the proof of Theorem \ref{Thm:TailPolya} is the following result, which bounds certain specific moments, depending on $y$, of the remaining part of the sum in \eqref{Eq:TailPolyaTwisted} over small integers $1\leq n\leq \exp((\log_2 Q)^2)$, in the special case where $a_n=h(n)$ if $P^{+}(n)>y$, and equals $0$ otherwise.
\begin{pro}\label{Pro:TailUnsmooth} Let $h$ be a completely multiplicative function such that $|h(n)|\leq 1$ for all integers $n\geq 1$. Let  
$X= \exp((\log_2 Q)^2)$. There exist  positive constants $C_9, C_{10}>1$ such that for any positive integer $C_9\leq k\leq \log Q/(10\log X)$ and real number $(k\log k)/10\leq y\leq k(\log k)^{10}$ we have 
$$
\sum_{\chi\in \F} \max_{t\in [0,1]}\max_{\alpha\in [0,1)} \Bigg|\sum_{\substack{1\leq n\leq X\\ P^{+}(n)> y}} \frac{\chi(n)h(n) e(n\alpha)}{n^{1+it}}\Bigg|^{2k} \ll Q \left(\frac{C_{10}k\log y}{y}\right)^{k}.
$$
\end{pro}

\begin{proof}
We will follow the proof of Proposition 4.5 of \cite{La22}, but there are certain simplifications along the way, since we are not aiming to optimize the constant $C_{10}$ in the upper bound. First, by a slight variation of  \cite[Eq. (4.10)]{La22} (where we add the twist by $n^{it}$) we have 
\begin{equation}\label{Eq:HolderSmooth}
\begin{aligned}
&\sum_{\chi\in \F} \max_{t\in [0,1]}\max_{\alpha\in [0,1)} \Bigg|\sum_{\substack{1\leq n\leq X\\ P^{+}(n)> y}} \frac{\chi(n) h(n)e(n\alpha)}{n^{1+it}}\Bigg|^{2k}\\
& \quad\quad \ll e^{O(k)}\left(\log y \right)^{2k-1}\sum_{\substack{1\leq a\leq X\\ P^{+}(a)\leq  y}}\frac{1}{a} \sum_{\chi\in \F}S_{\chi}(y, X/a)^{2k},
 \end{aligned}
\end{equation}
where 
\begin{equation}\label{Eq:DefinitionSchi}
S_{\chi}(y, z)= \max_{t\in [0,1]}\max_{\alpha \in [0, 1)}\Bigg|\sum_{\substack{1<n\leq z\\ P^{-}(n)> y}} \frac{\chi(n) h(n) e(n\alpha)}{n^{1+it}}\Bigg|.
\end{equation}
Therefore, it suffices to bound the moments 
$$ \sum_{\chi\in \F}S_{\chi}(y, z)^{2k}, $$
uniformly in $2\leq y<z\leq X$. Using Minkowski's inequality, we have 
\begin{equation}\label{Eq:MinkUnsmooth}
\begin{aligned}
\left(\sum_{\chi\in \F}S_{\chi}(y, z)^{2k}\right)^{1/2k} 
& \leq \Bigg(\sum_{\chi\in \F}\max_{t\in [0,1]}\max_{\alpha\in [0,1)} \Bigg|\sum_{\substack{1<n\leq N\\ P^{-}(n)> y}} \frac{\chi(n) h(n) e(n\alpha)}{n^{1+it}}\Bigg|^{2k}\Bigg)^{1/2k}\\
& \quad \quad + \Bigg(\sum_{\chi\in \F}\max_{t\in [0,1]}\max_{\alpha\in [0,1)} \Bigg|\sum_{\substack{N<n\leq z\\ P^{-}(n)> y}} \frac{\chi(n) h(n) e(n\alpha)}{n^{1+it}}\Bigg|^{2k}\Bigg)^{1/2k},
\end{aligned}
\end{equation}
where $
N:= \min\big(z, \exp((\log k)^2)\big).
$
Note that the second part will be empty unless $z>\exp\big((\log k )^2\big)$.
We start by bounding the first sum. By Lemma \ref{Lem:Reduction} with the choices $R:=\exp\big(2(\log k )^2\big)$ and 
\begin{equation}\label{Eq:Choicea_n}
a_n=\begin{cases}h(n) & \text{ if } P^-(n)>y,\\
0 &\text{ otherwise, }\end{cases}    
\end{equation}
 we obtain
\begin{equation}\label{Eq:Reduction2}
\begin{aligned}
&\Bigg(\sum_{\chi\in \F} \max_{t\in [0,1]}\max_{\alpha\in [0,1)} \Bigg|\sum_{\substack{1< n\leq N\\ P^{-}(n)>y}} \frac{\chi(n) h(n)e(n\alpha)}{n^{1+it}}\Bigg|^{2k}\Bigg)^{1/2k} \\
& \leq \Bigg(\sum_{t\in\T}\sum_{\alpha \in \mathcal{A}}  \sum_{\chi\in \F} \Bigg|\sum_{\substack{1< n\leq N^{k}\\ P^{-}(n)>y}}\frac{\chi(n)g_{2, N, k}(n, \alpha)}{n^{1+it}}\Bigg|^{2}\Bigg)^{1/2k}+ O\left(Q^{1/2k}R^{-{1/2}}\right).
\end{aligned}
\end{equation}
 We now proceed similarly to the proof of \eqref{Eq:Tail3}. Using Lemma \ref{Lem:Orthogonality} and noting that $N^k\leq Q^{1/10}$  we derive
\begin{equation}\label{Eq:TailDivisorSum0}
\begin{aligned}\sum_{\chi \in \F} \Bigg|\sum_{\substack{1< n\leq N^{k}\\ P^{-}(n)>y}}\frac{\chi(n)g_{2, N, k}(n, \alpha)}{n^{1+it}}\Bigg|^{2} &\ll Q\sum_{\substack{1<n_1, n_2 \leq N^{k}\\ P^{-}(n_1n_2)> y\\ n_1n_2^2 \textup{ is a cube}}}
\frac{d_k(n_1)d_k(n_2)}{n_1n_2} \\
& \quad \quad +O\left(Q^{2/3}\left(\sum_{n\leq N^k}\frac{|g_{2, N, k}(n, \alpha)|}{n}\right)^2\right)\\
&\ll Q\left(\frac{C_8k}{y\log y}\right)^{k},
\end{aligned}
\end{equation}
where the last inequality follows from Lemma \ref{Lem:DivisorRoughBig1} and the bound
$$\sum_{n\leq N^k}\frac{|g_{2, N, k}(n, \alpha)|}{n}\leq \left(\sum_{n\leq N}\frac{1}{n}\right)^k\ll (2\log N)^k=Q^{o(1)} $$
by our assumption on $k$.
Inserting these estimates in \eqref{Eq:Reduction2} implies that 
\begin{equation}\label{Eq:FirstTermHolder}
 \Bigg(\sum_{\chi\in \F} \max_{t\in [0,1]}\max_{\alpha\in [0,1)} \Bigg|\sum_{\substack{1< n\leq N\\ P^{-}(n)>y}} \frac{\chi(n) h(n) e(n\alpha)}{n^{1+it}}\Bigg|^{2k}\Bigg)^{1/2k} \quad \ll Q^{1/(2k)} \left(\frac{k}{y\log y}\right)^{1/2}.
\end{equation}

We now assume that $z>\exp\big((\log k)^2\big)$ and bound the second term on the right-hand side of \eqref{Eq:MinkUnsmooth}. We shall
split the inner sum over $n$ into dyadic intervals. Let $J_1= \lfloor \log N/\log 2\rfloor, J_2= \lfloor \log z/\log 2\rfloor$, and define $t_{J_1}= N$, $t_{J_2+1}= z$, and $t_j:= 2^j$ for $J_1+1\leq j\leq J_2$. Using Minkowski's inequality and Lemma \ref{Lem:TailMedium} with the same choice of $a_n$ as in \eqref{Eq:Choicea_n}
we derive
\begin{align*}
&\Bigg(\sum_{\chi\in \F}\max_{t\in [0,1]}\max_{\alpha\in [0,1)} \Bigg|\sum_{\substack{N<n\leq z\\ P^{-}(n)> y}} \frac{\chi(n) h(n) e(n\alpha)}{n^{1+it}}\Bigg|^{2k}\Bigg)^{1/2k}\\
& \leq \sum_{J_1\leq j\leq  J_2}\Bigg(\sum_{\chi\in \F}\max_{t\in [0,1]}\max_{\alpha\in [0,1)} \Bigg|\sum_{t_j<n\leq t_{j+1}} \frac{\chi(n) a_n e(n\alpha)}{n^{1+it}}\Bigg|^{2k}\Bigg)^{1/2k}\\
& \ll Q^{1/2k}\sum_{J_1\leq j\leq  J_2}  \exp\left(-\frac{j}{4\log k}\right)\ll Q^{1/2k} (\log k)\exp\left(-\frac{\log N}{4\log k} \right) \ll \frac{Q^{1/2k}}{k^{1/5}}.
\end{align*}
 Combining the above estimate with \eqref{Eq:MinkUnsmooth} and \eqref{Eq:FirstTermHolder} we obtain 
$$ \left(\sum_{\chi\in \F}S_{\chi}(y, z)^{2k}\right)^{1/2k} \ll Q^{1/2k}  \left(\frac{k}{y\log y}\right)^{1/2},$$
uniformly for $y<z\leq X$. Inserting this bound in \eqref{Eq:HolderSmooth} completes the proof. 
 \end{proof}

 We now complete the proof of Theorem \ref{Thm:TailPolya} combining Lemma \ref{Lem:TailLarge}, Lemma \ref{Lem:TailMedium}, and Proposition \ref{Pro:TailUnsmooth}.
 
 \begin{proof}[Proof of Theorem \ref{Thm:TailPolya}] Recall that $X= \exp((\log_2Q)^2)$ and $Z=Q^{21/40}$. Let $\delta=1/1000$. First by Lemma \ref{Lem:TailLarge} with the choice 
 \begin{equation}\label{Eq:SecondChoicea_n}
a_n=\begin{cases}h(n) & \text{ if } P^+(n)>y,\\
0 &\text{ otherwise, }\end{cases}    
\end{equation}
 we have 
\begin{equation}\label{Eq:BoundTailLarge}
\max_{t\in [0,1]}\max_{\alpha\in [0, 1)}\left|\sum_{\substack{Q^{\delta}\leq n\leq  Z\\ P^+(n)>y}} \frac{\chi(n) h(n)e(n\alpha)}{n^{1+it}}\right|\leq \frac{1}{\log Q},
 \end{equation}
 for all cubic characters $\chi\in \F$ except for a set of size $O(Q^{1-\delta}).$
 
 We now handle the remaining range $1\leq n\leq Q^{\delta},$ which we will split further in several pieces. As in the proof of Lemma \ref{Lem:TailLarge} we let  $L_1= \lfloor \log Q/(1000\log 2)\rfloor$, and put $L_0:= \lfloor \log X/\log 2\rfloor$. We also define $s_{L_0}:= X$, $s_{L_1+1}:= Q^{\delta}$, and we put $s_{\ell}:=2^{\ell}$ for $L_0+1\leq \ell\leq L_1$.
Let $\ell \in [L_1, L_2]$ and put $k=\lfloor\log Q /(10 \log s_{\ell})\rfloor$. Then using Lemma \ref{Lem:TailMedium} with the same choice of $a_n$ as in \eqref{Eq:SecondChoicea_n}, we obtain that the number of characters $\chi \in \F$ such that 
$$\max_{t\in [0,1]}\max_{\alpha\in [0,1)} \Bigg|\sum_{\substack{s_{\ell}\leq n\leq  s_{\ell+1}\\ P^+(n)>y}} \frac{\chi(n) h(n) e(n\alpha)}{n^{1+it}}\Bigg| \geq \frac{1}{\ell^2}$$ 
is 
\begin{align*}
& \leq \ell^{4k} \sum_{\chi\in \F} \max_{t\in [0,1]}\max_{\alpha\in [0,1)} \Bigg|\sum_{\substack{s_{\ell}\leq n\leq  s_{\ell+1}\\ P^+(n)>y}} \frac{\chi(n) h(n) e(n\alpha)}{n^{1+it}}\Bigg|^{2k} \\
&
\ll  2^{4\ell}Q\exp\left(-\frac{2k\log s_{\ell}}{\log k} + O(k\log_2 k +k \log \ell)\right) 
\ll Q  \exp\left(-\frac{\log Q}{20 \log_2 Q}\right).
\end{align*}
This implies that
\begin{equation}\label{Eq:BoundTailMedium}
\sum_{L_1\leq \ell \leq L_2} \max_{t\in [0,1]}\max_{\alpha \in [0, 1)} \Bigg|\sum_{\substack{s_{\ell}\leq n\leq  s_{\ell+1}\\ P^+(n)>y}} \frac{\chi(n) h(n) e(n\alpha)}{n^{1+it}}\Bigg| 
\ll  
\sum_{L_1\leq \ell \leq L_2} \frac{1}{\ell^2} \ll \frac{1}{(\log_2 Q)^2},
\end{equation}
for all  characters $\chi\in \F$ except for a set of size $Q  \exp\left(-\log Q/(30 \log_2 Q)\right)$. 

Let $J=\lfloor y/(20C_{10} \log y)\rfloor$, where $C_{10}$ is the constant in the statement of Proposition  \ref{Pro:TailUnsmooth}. Then it follows from this result that the number of characters $\chi\in \F$ such that 
$$ \max_{t\in [0,1]}\max_{\alpha \in [0, 1)} \Bigg|\sum_{\substack{1\leq n\leq X\\ P^{+}(n)> y}} \frac{\chi(n) h(n) e(n\alpha)}{n^{1+it}}\Bigg| > \frac12,$$
is 
$$\leq 2^{2J}\sum_{\chi\in \F}\max_{\alpha \in [0, 1)} \Bigg|\sum_{\substack{1\leq n\leq X\\ P^{+}(n)> y}} \frac{\chi(n) h(n) e(n\alpha)}{n^{1+it}}\Bigg|^{2J}\ll Q \left(\frac{4C_{10}J\log y}{y}\right)^{J}\ll Qe^{-J}.
$$
Combining this estimate with  \eqref{Eq:BoundTailLarge} and \eqref{Eq:BoundTailMedium} and noting that 
\begin{align*}
\max_{t\in [0,1]}\max_{\alpha \in [0, 1)} \Bigg|\sum_{\substack{1\leq n\leq Z\\P^+(n)>y}} \frac{\chi(n) h(n) e(n\alpha)}{n^{1+it}}\Bigg| 
&  \leq  \max_{t\in [0,1]}\max_{\alpha \in [0, 1)} \Bigg|\sum_{\substack{1\leq n\leq X\\ P^{+}(n)> y}} \frac{\chi(n) h(n) e(n\alpha)}{n^{1+it}}\Bigg| \\
& + \sum_{L_0\leq \ell \leq L_1} \max_{t\in [0,1]}\max_{\alpha \in [0, 1)} \Bigg|\sum_{\substack{s_{\ell}\leq n\leq s_{\ell+1}\\ P^+(n)>y}} \frac{\chi(n) h(n) e(n\alpha)}{n^{1+it}}\Bigg|\\
&+\max_{t\in [0,1]}\max_{\alpha \in [0, 1)} \Bigg|\sum_{\substack{Q^{\delta}\leq n\leq Z\\P^+(n)>y}} \frac{\chi(n) h(n) e(n\alpha)}{n^{1+it}}\Bigg|
\end{align*}
completes the proof.  
\end{proof}


\section{Distribution of $L(1, \chi\overline{\psi})$: Proof of Theorem \ref{Thm: L(1, chipsi)}} \label{Sec: lower}

To prove Theorem \ref{Thm: L(1, chipsi)} we will compute the moments $|L(1, \chi\overline{\psi})|^{2r}$ as $\chi$ varies in $\mathcal{F}_3(Q)$. Using the results of \cite{DDLL23} we shall prove that these moments are close to the corresponding moments of a certain random Euler product. Let $\{\mathbb{Y}(p)\}_{p\:\text{prime}}$ be independent random variables,  defined by 
$$ \Y(p)=\begin{cases} 1 & \text{ with probability } 1/3,\\
\omega_3 & \text{ with probability } 1/3,\\
\omega_3^2 & \text{ with probability } 1/3,\\
\end{cases}
$$
when $p=3$ or $p\equiv 2\pmod 3$, and by 
$$ \Y(p)=\begin{cases} 0 & \text{ with probability } 2/(p+2),\\
1 & \text{ with probability } p/(3(p+2)),\\
\omega_3 & \text{ with probability } p/(3(p+2)),\\
\omega_3^2 & \text{ with probability } p/(3(p+2)),\\
\end{cases}
$$
 when $p\equiv 1\pmod 3$. Then, for any Dirichlet character $\psi$ modulo $m$, we define
\begin{align}\label{Def: L(1, ypsi)}
    L(1, \mathbb{Y}\overline{\psi}):=\prod_{p\: \text{prime}}\bigg(1-\dfrac{\mathbb{Y}(p)\overline{\psi(p)}}{p}\bigg)^{-1}.
\end{align}
We observe that this random Euler product converges almost surely by Kolmogorov's two-series Theorem since $\ex(\Y(p))=0$ for all primes $p$.

Our first result, which is a minor modification of \cite[Theorem 3.1]{DDLL23}, compares the moments $|L(1, \chi\overline{\psi})|^{2r}$ for $\chi\in \mathcal{F}_3(Q)$ to the expectations $\ex (|L(1, \Y\overline{\psi})|^{2r})$.

\begin{pro}\label{Pro:CubicL-functions}
    Let $Q$ be large. Let $\psi$ be a primitive Dirichlet character modulo a non-exceptional modulus $m$ with $m\leq \log Q$. Then uniformly for all real numbers $r$ in the range $|r|\leq (\log Q)/(16 \log_2 Q\log _3 Q)$ we have 
\begin{align*}
        \dfrac{1}{\# \mathcal{F}_3(Q)}\sum_{\chi \in \mathcal{F}_3(Q)}|L(1, \chi \overline{\psi})|^{2r} = \ex (|L(1, \Y \overline{\psi})|^{2r}) + O\bigg(\exp \bigg(-\dfrac{\log Q}{16\log_2 Q}\bigg)\bigg).
    \end{align*}
\end{pro}

\begin{proof}
The proof is an easy adaptation of the proof of Theorem 3.1 of \cite{DDLL23}. Indeed, this is the case since the conductor of $\chi \overline{\psi}$ is $\leq Q\log Q$, so Lemma 3.2 of \cite{DDLL23} can be applied verbatim with the same choice of $\ep=1/6$ the authors used in the proof of their Theorem 3.1. The remaining part of the proof is exactly the same since $|\psi(n)|\leq 1$ for all $n$.
\end{proof}

In order to use the above proposition, we need to estimate the expectation $\ex (|L(1, \Y \overline{\psi})|^{2r})$, which we do in the following result.

\begin{pro}\label{Pro:RandomL-functions}
    Let $r>1$ be a large real number. Let $\psi$ be a primitive Dirichlet character modulo a non-exceptional modulus $m$ with 
    $m\leq \log Q$. Let  $k$ be the order of $\psi$. Then, we have
    \begin{align*}
        \ex (|L(1, \Y \overline{\psi})|^{2r}) = \exp\bigg(2r\beta \log_2 r +  O\bigg(r\dfrac{\log_2 r}{k^2} + r\log_2m\bigg)\bigg).
    \end{align*}
\end{pro}

\begin{proof}
By \eqref{Def: L(1, ypsi)} and the independence of  the random variables $\{\Y(p)\}_{p\:\text{prime}}$, we have
   
\begin{equation}
    \label{Eq:RandomL1}
\begin{aligned}
     \ex (|L(1, \Y \overline{\psi})|^{2r})&=\prod_{p\: \text{prime}}\ex \bigg
    (\bigg|1-\frac{\Y(p)\overline{\psi (p)}}{p}\bigg|^{-2r}\bigg)\\
    &=\prod_{p\: \text{prime}}\ex \left(\bigg(1-\dfrac{2\re (\Y(p)\overline{\psi(p)})}{p} + O
\left(\dfrac{1}{p^2}\right)\bigg)^{-r}\right).
\end{aligned}
\end{equation}


We will split the above product into two parts, depending on whether $p$ is small or large. First, if $p>10r$ 
we have
\begin{align*}
    \ex \bigg(\bigg(1-\frac{2\re (\Y(p)\overline{\psi(p)})}{p} + O\bigg(\dfrac{1}{p^2}\bigg)\bigg)^{-r}\bigg) &=\ex \bigg(1 + \dfrac{2r \re (\Y(p)\overline{\psi(p))}}{p} + O\bigg(\dfrac{r^2}{p^2}\bigg)\bigg)\\
    &=1 +  O\bigg(\dfrac{r^2}{p^2}\bigg),
\end{align*}
since $\ex (\Y(p))=0$. Therefore, we get
\begin{align}
   \notag  \prod_{p>10r }\ex &\bigg(\bigg(1-\frac{2\re \Y(p)\overline{\psi(p)}}{p} + O\bigg(\dfrac{1}{p^2}\bigg)\bigg)^{-r}\bigg)\\
   \label{Eq:RandomLlargeprimes} &=\prod_{p> 10r }\bigg(1 + O\bigg(\dfrac{r^2}{p^2}\bigg)\bigg)=\exp\bigg(O\bigg(r^2\sum_{p>10r }\dfrac{1}{p^2}\bigg)\bigg)
    = \exp\bigg(O\bigg(\dfrac{r}{\log r}\bigg)\bigg)
\end{align}
by the Prime Number Theorem and partial summation.

Next, we deal with the small primes. If $p\leq  10r $, then we have
\begin{align}
\label{Eq:RandomLsmall1}  \ex \bigg(\bigg(1-\frac{2\re (\Y(p)\overline{\psi(p)})}{p} + O\bigg(\dfrac{1}{p^2}\bigg)\bigg)^{-r}\bigg) &=\ex \bigg(\exp\bigg(\dfrac{2r \re (\Y(p)\overline{\psi(p)})}{p} + O\bigg(\dfrac{r}{p^2}\bigg)\bigg)\bigg).
\end{align}
Next, we observe that
\begin{align}
 \notag    \prod_{p\leq  10r }\ex &\bigg(\exp\bigg(\dfrac{2r\re (\Y(p)\overline{\psi(p)})}{p}\bigg)\bigg)\\
    \label{Eq:RandomLsmall2} &\leq \exp\bigg(2r\sum_{\ell \bmod k }\max_{z\in \{0, 1, \omega_3, \omega_3^2\}}\re \bigg(z\cdot e\bigg(-\dfrac{\ell}{k}\bigg)\bigg)\sum_{\substack{p\leq 10r \\ \psi(p)=e(\ell/k)}}\dfrac{1}{p}\bigg).
\end{align}
By \cite[Proposition 6.1]{LaMa22}, we have
\begin{equation}\label{Eq:RandomLsmall3}
\begin{aligned}   \sum_{\ell \bmod k }\max_{z\in \{0, 1, \omega_3, \omega_3^2\}} &\re \bigg(z\cdot e\bigg(-\dfrac{\ell}{k}\bigg)\bigg)\sum_{\substack{p\leq 10r \\ \psi(p)=e(\ell/k)}}\dfrac{1}{p}\\
  & = \beta \cdot \dfrac{\pi (3, k)/(3k)}{\tan (\frac{\pi}{3k/(3,k)})} \log \log r+ O(\log \log m)\\
    &=\beta \log \log r + O\bigg(\dfrac{\log \log r}{k^2} + \log \log m\bigg),
\end{aligned}
\end{equation}
where we have used the fact that $t/(\tan t)=1 +O(t^2)$. Therefore, combining \eqref{Eq:RandomL1}, \eqref{Eq:RandomLlargeprimes}, \eqref{Eq:RandomLsmall1}, \eqref{Eq:RandomLsmall2}, \eqref{Eq:RandomLsmall3} we can conclude that
\begin{align*}
  \ex (|L(1, \Y \overline{\psi})|^{2r}) \leq \exp\bigg(2r\beta \log_2r + O\bigg(r\dfrac{\log_2r}{k^2} + r\log_2m\bigg)\bigg),
\end{align*}
which gives the desired upper bound.

Let us now turn our attention to the lower bound. First, we note that for all $\theta\in [0,1]$, the maximum of $\re(ze(-\theta))$  for $z\in \{0, 1, \omega_3, \omega_3^2\}$ is attained at one of the points $1, \omega_3$, or $\omega_3^2$. We choose $\Y(p)=z_\ell\in \{1, \omega_3, \omega_3^2\}$ optimally with probability $1/3$ or $p/(3(p+2))$ depending on whether $p=3$, $p\equiv 2\pmod 3$ or $p\equiv 1\pmod 3$. Hence, by the independence of the $\Y(p)$, we get
\begin{align*}
     &\prod_{p\leq 10r}\ex \bigg(\exp\bigg(\dfrac{2r\re (\Y(p)\overline{\psi(p)})}{p}\bigg)\bigg)\\
     & \geq \prod_{p\leq 10r}\left(\frac{p}{3(p+2)}\right)\exp\bigg(2r\sum_{\ell \bmod k}\re \bigg(z_\ell \cdot e\bigg(-\dfrac{\ell}{k}\bigg)\bigg)\sum_{\substack{p\leq 10r\\ \psi(p)=e(\ell/k)}}\dfrac{1}{p}\bigg).
\end{align*}
Note that by the Prime Number Theorem, $\prod_{p\leq 10r }\left(p/(3(p+2))\right) =\exp\left(O\left(r/\log r\right)\right).$
Therefore, by \eqref{Eq:RandomLsmall3} we obtain 
    \begin{align}\label{Eq:RandomLlower}
     \prod_{p\leq  10r }\ex \bigg(\exp\bigg(\dfrac{2r\re (\Y(p)\overline{\psi(p)})}{p}\bigg)\bigg)\geq \exp\bigg(2r\beta \log_2r + O\bigg(r\dfrac{\log_2r}{k^2} + r\log_2m\bigg)\bigg).
\end{align}
Combining the above estimate \eqref{Eq:RandomLlower} together with \eqref{Eq:RandomL1}, \eqref{Eq:RandomLlargeprimes}, \eqref{Eq:RandomLsmall1} completes the proof of the lower bound.
\end{proof}

We are now ready to complete the proof of Theorem \ref{Thm: L(1, chipsi)} by combining Propositions \ref{Pro:CubicL-functions} and \ref{Pro:RandomL-functions}.

\begin{proof}[Proof of Theorem \ref{Thm: L(1, chipsi)}] Recall that we wish to show that
\begin{align*}
        \dfrac{1}{\# \mathcal{F}_3(Q)}\#\bigg\{\chi\in \mathcal{F}_3(Q)\colon 
        |L(1, \chi\overline{\psi})|>V\bigg\}= \exp\left(-\exp\left(V^{1/\beta}(\log_2 V)^{O(1)}\right)\right). 
    \end{align*}
We first begin with the lower bound. Let $1\leq r\leq r_{\max}:=\log Q/(32\log_2Q\log_3Q)$ be a real number to be chosen appropriately so that
\begin{equation}\label{Eq:ConditionrV}
V^{2r}=\dfrac{1}{2}S(r),
\end{equation}
where 
$$S(u):=\dfrac{1}{\#\mathcal{F}_3(Q)} \sum_{\chi\in \mathcal{F}_3(Q)}|L(1, \chi\overline{\psi})|^{2u}.$$
Then, by the Paley-Zygmund inequality, we have
\begin{equation}\label{Eq:PaleyZygmund}
\begin{aligned}
        \dfrac{1}{\# \mathcal{F}_3(Q)}\#\bigg\{\chi\in \mathcal{F}_3(Q)\colon |L(1, \chi\overline{\psi})|>V\bigg\} \geq \frac14 \frac{S(r)^2}{S(2r)}=\dfrac{V^{4r}}{S(2r)}.
    \end{aligned}
\end{equation}
For all $1\leq u\leq  r_{\max}$
it follows from Propositions \ref{Pro:CubicL-functions} and \ref{Pro:RandomL-functions} that
\begin{equation}\label{Eq:EstimaterV}
    \begin{aligned}
        S (u)&=\ex (|L(1, \Y\overline{\psi})|^{2u}) + O\bigg(\exp\bigg(-\dfrac{\log Q}{16\log_2 Q}\bigg)\bigg)\\
    &= \exp\bigg(2u\beta \log_2 u + O\bigg(\dfrac{u\log_2 u}{k^2} + u\log_2m\bigg)\bigg),
    \end{aligned}
\end{equation}
where $k$ is the order of $\psi$.
Note that by our assumption, $m\in [\sqrt{\log V}, 2\sqrt{\log V}]$ and $k=m-1$. This implies $k\asymp \sqrt{\log V}$. Furthermore, if $r$ verifies \eqref{Eq:ConditionrV} then \eqref{Eq:EstimaterV} shows that we must have $\log V\asymp \log_2 r$, and hence
\begin{equation}\label{Eq:EstimateSr}
S(u)=\exp\bigg(2u\beta\log_2u +O(u\log_3V)\bigg),  
\end{equation} uniformly for all $u\in [r/2, 3r].$
Therefore, by our assumption on $V$ we deduce that an $r\in [1, r_{\max}]$ verifying \eqref{Eq:ConditionrV} must exist, since both sides of \eqref{Eq:ConditionrV} are continuous functions of $r$, the left hand side is larger than the right when $r=1$, and the reverse is true when $r=r_{\textup{max}}.$ Moreover, we infer from \eqref{Eq:ConditionrV} and \eqref{Eq:EstimateSr} that
$$
\log V= \beta\log_2 r+O(\log_4r).
$$
and hence
$$
r= \exp\left(V^{1/\beta}(\log_2 V)^{O(1)}\right).
$$
Inserting these estimates in \eqref{Eq:PaleyZygmund} implies
\begin{align*}
\dfrac{1}{\# \mathcal{F}_3(Q)}\#\bigg\{\chi\in \mathcal{F}_3(Q)\colon |L(1, \chi\overline{\psi})|>V\bigg\} \geq \dfrac{V^{4r}}{S(2r)} &\gg \exp\left(-C_{11}r\log_4r\right)\\
&\gg \exp\left(-\exp\left(V^{1/\beta}(\log_2 V)^{O(1)}\right)\right),
    \end{align*}
for some positive constant $C_{11}$. This gives the desired lower bound.

Finally, for the upper bound, we use \eqref{Eq:EstimaterV} with the choice $u= \exp\left(V^{1/\beta}(\log_2 V)^{-C_{12}}\right)$, for a suitably large constant $C_{12}$, to obtain
\begin{align*}
\dfrac{1}{\# \mathcal{F}_3(Q)}\#\bigg\{\chi\in \mathcal{F}_3(Q)\colon |L(1, \chi\overline{\psi})|>V\bigg\}   \leq  \dfrac{S(u)}{V^{2u}} &\ll \exp\left(-C_{12}u\log_4 u\right)\\
&\ll  \exp\left(-\exp\left(V^{1/\beta}(\log_2 V)^{O(1)}\right)\right).
\end{align*}
This completes the proof.

\end{proof}


\section{A structure theorem for large cubic character sums: Proof of Theorem \ref{Thm:Structure}}\label{Sec: structure}
\begin{proof}[Proof of Theorem \ref{Thm:Structure}]
We define $\mathcal{C}_Q(V)$ to be the set of cubic characters $\chi\in \F$ such that $M(\chi)>V$
and \begin{equation}\label{Eq:StrongAssumptionTail}
 \max_{t\in [0,1]}\max_{\alpha\in [0,1)}\Bigg|\sum_{\substack{1\leq n\leq Z\\ P^{+}(n)> y}} \frac{\chi(n) e(n\alpha)}{n^{1+it}}\Bigg|<1,   
\end{equation}
where $Z=Q^{21/40}$ and $y=y(V)$ is given by \eqref{Eq:DefinitionYV}.
Then it follows from Theorems \ref{Thm:Main} and \ref{Thm:TailPolya} that 
$$ \# \mathcal{C}_Q(V)= \left(1+O\left(e^{-\frac{y}{(\log y)^2}}\right)\right)\#\left\{ \chi\in \F \colon M(\chi)>V \right\}.$$
Let $\chi\in \mathcal{C}_Q(V).$ By \eqref{Eq:PolyaFourier2} and our assumption \eqref{Eq:StrongAssumptionTail} we have 
\begin{equation}\label{Eq:MChiAlphaChi}
   M(\chi)=\frac{1}{2\pi}\Bigg|\sum_{\substack{1\leq |n|\leq Z\\ P^+(n)\leq y}} \frac{\chi(n)e(n\alpha_{\chi})}{n}\Bigg| +O(1).
\end{equation}
If $b>L=(\log y)^{4/11}$ then by \cite[Corollary 2.2]{Gol}, we have
\begin{equation}\label{Eq:ApproxAlphaChi}
\begin{aligned}
\sum_{\substack{1\leq |n|\leq Z\\ P^+(n)\leq y}}\dfrac{\chi(n)e(n\alpha_{\chi})}{n}\ll \log B +  \dfrac{(\log L)^{5/2}}{L^{1/2}}(\log y) + \log_2y\ll \frac{V}{\log V} ,
\end{aligned}
\end{equation}
contradicting the fact that $M(\chi)>V.$ Therefore, we must have $b\leq (\log y)^{4/11}.$ We now apply \cite[Lemma 4.1]{Gol} to obtain
\begin{align*}
\sum_{\substack{1\leq |n|\leq Z\\ P^+(n)\leq y}}\dfrac{\chi(n)e(n\alpha_{\chi})}{n}
&=\sum_{\substack{1\leq |n|\leq K\\ P^+(n)\leq y}}\dfrac{\chi(n)e(na/b)}{n} + O\bigg(\dfrac{(\log B)^{3/2}}{B^{1/2}}(\log y)^2 + \log B + \log_2y\bigg)\\
&= \sum_{\substack{1\leq |n|\leq K\\ P^+(n)\leq y}}\dfrac{\chi(n)e(na/b)}{n} + O\bigg(\frac{V}{\log V}\bigg),
\end{align*}
where $K=\min(Z, |b\alpha_{\chi}-a|^{-1}).$ Therefore, we deduce from \eqref{Eq:MChiAlphaChi} that 
$$\Bigg|\sum_{\substack{1\leq |n|\leq K\\ P^+(n)\leq y}}\dfrac{\chi(n)e(na/b)}{n}\Bigg|\gg V, $$ and hence by Lemma \ref{Lem: major1}  we must have $m\mid b$ and $\xi$ is odd, since $9/(11\beta)<1$. Moreover, combining the estimates \eqref{Eq:MChiAlphaChi}, \eqref{Eq:ApproxAlphaChi} with Lemma \ref{Lem: major1} we deduce in this case
\begin{equation}\label{Eq:UpperMChiWithb}
\begin{aligned}M(\chi)&\ll \dfrac{m^{3/2}}{\varphi(m) b}\prod_{p\mid \frac{b}{m}}\left(1+\frac{p+1}{p-1}\right)(\log y)e^{-\mathcal{M}(\chi\overline{\xi}; y, (\log y)^{-7/11})}\\
&\ll \dfrac{1}{ b^{1/2+o(1)}}(\log y)e^{-\mathcal{M}(\chi\overline{\xi}; y, (\log y)^{-7/11})},
\end{aligned}
\end{equation}
since $\varphi(m)\gg m/\log_2m$ and 
$\prod_{p\mid \frac{b}{m}}\left(1+\frac{p+1}{p-1}\right)\leq 4^{\omega(b)}, $
where $\omega(b)$ is the number of distinct prime factors of $b$, and where we have used that $\omega(b)=o(\log b).$
Next, by Lemma \ref{Lem: major2} we have
\begin{equation}\label{Eq:LowerBoundMStructure}\mathcal{M}(\chi\overline{\xi}; y, (\log y)^{-7/11})\geq (1-\beta)\log_2 y+o(\log m).
\end{equation}
Therefore, we deduce that 
$$ V<M(\chi)\ll b^{-1/2+o(1)}(\log y)^{\beta}= b^{-1/2+o(1)} V(\log V)^{1/4+O(\ep)}.$$
This shows that 
\begin{equation}\label{Eq:UpperBoundb}
b\leq (\log V)^{1/2+O(\ep)}.
\end{equation}
Furthermore, by \eqref{Eq:UpperMChiWithb} and Lemma \ref{Lem: major2} we obtain 
\begin{equation}\label{Eq:OrderMChiLarge}
m^{1/2+o(1)}V<m^{1/2+o(1)}M(\chi) \ll (\log y)e^{-\mathcal{M}(\chi\overline{\xi}; y, (\log y)^{-7/11})}\ll (\log y)^{\beta-c_0/m^2} m^{o(1)},
\end{equation}
for some positive constant $c_0.$ In particular, this implies 
$$\exp\left(\left(\frac12+o(1)\right)\log m+ c_1\frac{\log V}{m^2}\right)\leq (\log V)^{1/4+O(\ep)},$$
for some positive constant $c_1.$
Writing $u=m/\sqrt{\log V}$ and taking the logarithm on both sides, we derive
    $$\left(\frac{1}{2}+o(1)\right)\log u+ \frac{c_1}{u^2}\leq  O(\ep\log_2 V).
    $$
This shows that $u=(\log V)^{O(\ep)}$ and consequently $m=(\log V)^{1/2+O(\ep)}$ and $b=(\log V)^{1/2+O(\ep)}$ by \eqref{Eq:UpperBoundb}, which completes the proof of Parts 1 and 2. 

Next, we establish Part 3. Inserting the estimate on $m$ in \eqref{Eq:UpperMChiWithb} and using \eqref{Eq:LowerBoundMStructure} we derive 
$$ V<M(\chi) \ll \frac{1}{m^{1/2+o(1)}}(\log y)e^{-\mathcal{M}(\chi\overline{\xi}; y, (\log y)^{-7/11})}\ll V(\log V)^{O(\ep)}.$$
Thus, we deduce that
\begin{equation}\label{Eq:TrueOrderMChi}
M(\chi)= \frac{1}{m^{1/2+O(\ep)}} (\log y) \exp\left(-\mathcal{M}(\chi\overline{\xi}; y, (\log y)^{-7/11})\right).
\end{equation} Furthermore, by our definition of $t_{\chi}$ we have 
$$
(\log y) \exp\left(-\mathcal{M}(\chi\overline{\xi}; y, (\log y)^{-7/11})\right)\asymp   \exp\left(\re\sum_{p\leq y} \frac{\chi(p)\overline{\xi(p)}}{p^{1+it_{\chi}}}\right)\asymp\prod_{p\leq y}\left|1-\frac{\chi(p)\overline{\xi(p)}}{p^{1+it_{\chi}}}\right|^{-1}.
$$
Moreover, since the conductor of $\chi\overline{\xi}$ is $\leq Q\log Q$, it follows from the P\'olya-Vinogradov inequality that 
\begin{equation}\label{Eq:ApproxL1chiPV}
L(1+it_{\chi},\chi\overline{\xi})= \sum_{n\leq Z}\frac{\chi(n)\overline{\xi}(n)}{n^{1+it_{\chi}}}+O(1),
\end{equation}
where $Z=Q^{21/40}$ as before.
Furthermore, by our assumption \eqref{Eq:StrongAssumptionTail} we have 
$$\sum_{\substack{1\leq n\leq Z\\ P^{+}(n)> y}}\frac{\chi(n)\overline{\xi}(n)}{n^{1+it_{\chi}}}= \frac{1}{\tau(\xi)}\sum_{a\bmod m}\xi(a)\sum_{\substack{1\leq n\leq Z\\ P^{+}(n)> y}}\frac{\chi(n)e(an/m)}{n^{1+it_{\chi}}}\ll \sqrt{m}\ll \log V.$$
Inserting this bound in \eqref{Eq:ApproxL1chiPV}
we deduce that 
$$ L(1+it_{\chi},\chi\overline{\xi})=\sum_{\substack{1\leq n\leq Z\\ P^{+}(n)\leq  y}}\frac{\chi(n)\overline{\xi}(n)}{n^{1+it_{\chi}}} +O(\log V)= \prod_{p\leq y}\left(1-\frac{\chi(p)\overline{\xi(p)}}{p^{1+it_{\chi}}}\right)^{-1}+O(\log V),$$
since 
$$ \sum_{\substack{n> Z\\ P^{+}(n)\leq  y}}\frac{1}{n}\ll 1,$$
by \cite[Lemma 3.2]{BGGK18}.
This completes the proof.
\end{proof}

\section*{Acknowledgements}
The authors would like to thank the anonymous referee for a careful reading of the manuscript and valuable comments.


\bibliographystyle{plain}

\end{document}